\documentclass[10pt, a4paper]{article}
\usepackage{titlesec} 
\usepackage{amsfonts}
\usepackage{amsthm}
\usepackage{amssymb}
\usepackage{amsmath}
\usepackage{theoremref}
\usepackage{enumerate}
\usepackage{bbm}
\usepackage{authblk}
\usepackage{a4wide}

\makeatletter
\g@addto@macro\bfseries{\boldmath}
\makeatother

\newcommand{\A}{\mathcal{A}}
\newcommand{\C}{\mathcal{C}}
\newcommand{\M}{\mathcal{M}}
\newcommand{\T}{\mathbb{T}}
\newcommand{\z}{\zeta}
\newcommand{\conj}[1]{\overline{#1}}

\newcommand{\D}{\mathbb{D}}

\newcommand{\cD}{\text{clos}(\mathbb{D})}
\newcommand{\ip}[2]{\big\langle #1, #2 \big\rangle}

\newcommand{\m}{\textit{m}}
\newcommand{\hil}{\mathcal{H}}

\newcommand{\deh}{\text{clos}({\Delta H^2(Y)})}
\newcommand{\h}{\mathcal{H}}

\newcommand{\bhk}{\mathcal{B}(Y,X)}
\newcommand{\bx}{\mathcal{B}(X)}

\renewcommand\Im{\operatorname{Im}}

\newtheorem{thm}{Theorem}[section]
\newtheorem{lemma}[thm]{Lemma}
\newtheorem{cor}[thm]{Corollary}
\newtheorem{prop}[thm]{Proposition}
\theoremstyle{definition}

\theoremstyle{definition}
\newtheorem{definition}{Definition}[section]

\titleformat{\section}
{\normalfont\scshape\centering}{\thesection}{1em}{}

\titleformat{\subsection}
[runin]{\bfseries\scshape}{\thesubsection}{0.5em}{}

\titleformat{\subsubsection}
[runin]{}{\thesubsubsection}{0.5em}{}

\begin{document}
\title{\textbf{Hilbert spaces of analytic functions with a contractive backward shift}}
\author{Alexandru Aleman, Bartosz Malman}
\date{}
\maketitle

\begin{abstract}
We consider Hilbert spaces of analytic functions in the disk with a normalized reproducing kernel and such that the backward shift $f(z) \mapsto \frac{f(z)-f(0)}{z}$ is a contraction on the space. We present a model for this operator and use it to prove the surprising result that functions which extend continuously to the closure of the disk are dense in the space. This has several applications, for example we can answer a question regarding reverse Carleson embeddings for these spaces. We also identify a large class of spaces which are similar to the de Branges-Rovnyak spaces and prove some results which are new even in the classical case. 
\end{abstract}

\section{Introduction}
Let $\D$ be the unit disk of the complex plane $\mathbb{C}$. The present paper is concerned with the study of a class of Hilbert spaces of analytic functions on which the backward shift operator acts as a contraction. More precisely, let $\h$ be a Hilbert space of analytic functions such that 

\begin{enumerate}[({A.}1)] \item \label{a11} the evaluation $f \mapsto f(\lambda)$ is a bounded linear functional on $\h$ for each $\lambda \in \D$,
	\item \label{a12} $\hil$ is invariant under the backward shift operator $L$ given by $$Lf(z) = \frac{f(z)-f(0)}{z},$$ and we have that $$\|Lf\|_{\h} \leq \|f\|_{\h}, \quad f \in \hil,$$ 
	\item \label{a13} the constant function 1 is contained in $\h$ and has the reproducing property $$\ip{f}{1}_\hil = f(0) \quad f \in \hil.$$ 
\end{enumerate}
Here $\ip{\cdot}{\cdot}_\h$ denotes the inner product in $\h$. By (A.1), the space $\h$ comes equipped with a reproducing kernel $k_\h: \D \times \D \to \mathbb{C}$ which satisfies $$\ip{f}{k_\h(\cdot, \lambda)}_\h = f(\lambda), f \in \h.$$ The condition (A.3) is a normalization condition which ensures that $k_\h(z,0) = 1$ for all $z \in \D$. The first example that comes to mind is a weighted version of the classical Hardy space, where the norm of an element $f(z) = \sum_{k=0}^\infty f_kz^n$ is given by $$\|f\|^2_w = \sum_{k=0}^\infty w_k|f_k|^2$$ with $w = (w_k)_{k=0}^n$ a nondecreasing sequence of positive numbers $w_k$, and $w_0 = 1$. More generally, the conditions are fulfilled in any $L$-invariat Hilbert space of analytic functions in $\D$ with a normalized reproducing kernel on which the forward shift operator $M_z$, given by $M_zf(z) = zf(z)$, is expansive ($\|M_zf\|_\h\ge\|f\|_\h$). Moreover, any $L$-invariant subspace of such a space of analytic functions satisfies the axioms as well if it contains the constants. A more detailed list of examples will be given in the next section. It includes de Branges-Rovnyak spaces, spaces of Dirichlet type and their $L$-invariant subspaces. Despite the conditions being rather general, it turns out that they imply useful common structural properties of these function spaces. The purpose of this paper is to reveal some of those properties and discuss their applications. 

The starting point of our investigation is the following formula for the reproducing kernel. The space $\h$ satisfies  (A.1)-(A.3)  if and only if the reproducing kernel $k_\h$ of $\h$ has the form \begin{equation}k_\h(z, \lambda) = \frac{1-\sum_{i\ge 1} \conj{b_i(\lambda)}b_i(z)}{1-\conj{\lambda}z} = \frac{1-\textbf{B}(z)\textbf{B}(\lambda)^*}{1-\conj{\lambda}z}, \quad \textbf{B}(0)=0, \label{kernelequation} \end{equation} 
where  $\textbf{B}$ is the analytic  row contraction  into $l^2$ with entries $(b_i)_{i \ge 1}$.  This follows easily from the positivity of the operator $I_\h - LL^*$ and will be proved in \thref{kernel} below. 


The representation in \eqref{kernelequation} continues to hold in the case when $\h$ consists of vector-valued functions, with $\textbf{B}$ an analytic operator-valued contraction.  Such kernels have been considered in \cite{ballbolotnikov} where they are called de Branges-Rovnyak kernels. This representation of the reproducing kernel in terms of $\textbf{B}$ is obviously not unique. We shall  denote throughout  by $[\textbf{B}]$ the  class of $\textbf{B}$ with respect to the equivalence $\textbf{B}_1\sim\textbf{B}_2\Leftrightarrow \textbf{B}_1(z)\textbf{B}_1^*(\lambda)=\textbf{B}_2(z)\textbf{B}_2^*(\lambda),~z,\lambda\in \D$, and by 
 $\h[\textbf{B}]$ the Hilbert space of analytic functions satisfying (A.1)-(A.3) for which the reproducing kernel is given by \eqref{kernelequation}. 
\begin{definition} \label{defhB}  The space $\h[\textbf{B}]$ is of \textit{finite rank} if  there exists $\textbf{C} \in [\textbf{B}],~\textbf{C}= (c_1, \ldots, c_n, \ldots)$ and an integer $N\ge 1$ with $c_n=0,~n\ge N$. The \textit{rank} of $\h[\textbf{B}]$ is  defined as the minimal number of nonzero terms which can occur in these representations. 
\end{definition}
Note that if  $\h[\textbf{B}]$ is of finite rank and $\textbf{C} \in [\textbf{B}],~\textbf{C}= (c_1, \ldots, c_n, \ldots)$ has the (nonzero) minimum number of nonzero terms, then these must be linearly independent.
 The rank-zero case corresponds to the Hardy space $H^2$, while rank-one spaces are the classical  de Branges-Rovnyak  spaces. These will be denoted by $\h(b)$.

Our basic tool for the study of these spaces is a model for the contraction $L$ on $\h$. The intuition comes from a simple example, namely a de Branges-Rovnyak space $\h(b)$ where $b$ is a non-extreme point of the unit ball of $H^\infty$. Then there exists an analytic outer function $a$ such that $|b|^2 + |a|^2 = 1$ on $\T$, and we can consider the $M_z$-invariant subspace $U = \{(bh,ah) : h \in H^2\} \subset H^2 \oplus H^2$. It is proved in \cite{sarasonbook}, Section IV-7, that there is an isometric one-to-one correspondence $f \mapsto (f,g)$ between the elements of $\h(b)$ and the tuples in the orthogonal complement of $U$, and the intertwining relation $Lf \mapsto (Lf, Lg)$ holds. 

One of the main ideas behind the results of this paper is that, based on the structure of the reproducing kernel, a similar construction can be carried out for any Hilbert space satisfying (A.1)-(A.3). As is to expect, in this generality the objects appearing in our model are more involved, in particular the direct sum $H^2 \oplus H^2$ needs to be replaced by the direct sum of $H^2$ with an $M_z$-invariant subspace of a vector-valued $L^2$-space. Details of this construction, which extend to the vector-valued case as well, will be given in Section 2. This model is essentially a special case of the functional model of Sz.-Nagy-Foias for a general contractive linear operator (see Chapter VI of \cite{nagyfoiasharmop}). Moreover, it has a connection to a known norm formula (see e.g. \cite{afr}, \cite{alemanrichtersimplyinvariant}, \cite{ars}) which played a key role in the investigation of invariant subspaces in various spaces of analytic functions. We explain this connection in Section \ref{formulasubsection}.

The advantage provided by this point of view is a fairly tractable formula for the norm on such spaces. Our main result in this generality is the following surprising approximation theorem. Recall that the disk algebra $\A$ is the algebra of analytic functions in $\D$ which admit continuous extensions to $\cD$.  

\begin{thm} \thlabel{thm1}
Let $\h$ be a Hilbert space of analytic functions which satisfies (A.1)-(A.3). Then any backward shift invariant subspace of $\h$ contains a dense set of functions in $\A$. 
\end{thm} 

The proof of \thref{thm1} is carried out in Section \ref{densitysection} and the argument covers the finite dimensional vector-valued case as well. Our approach is based on ideas of Aleksandrov \cite{aleksandrovinv} and the authors \cite{dbrcont}, however due to the generality considered here the proof is different since it  avoids the use of classical theorems of Vinogradov \cite{vinogradovthm} and Khintchin-Ostrowski \cite[Section 3.2]{havinbook}.

Several applications of \thref{thm1} are presented in Section \ref{applications1}. One of them concerns the case when $\h$  satisfies (A.1)-(A.3) and, in addition, it is invariant for the forward shift. We obtain a very general Beurling-type theorem for $M_z$-invariant subspaces which requires the use of  \thref{thm1} in the vector-valued case. 

\begin{cor} \thlabel{cor-beurling1} Let $\h$ be a Hilbert space of analytic functions which satisfies (A.1)-(A.3)  and is invariant for the forward shift $M_z$. For a closed $M_z$-invariant subspace $\M$ of $\h$ with $\dim \M\ominus M_z\M=n<\infty$, let $\varphi_1,\ldots\varphi_n$ be an orthonormal basis in $\M\ominus M_z\M$, and denote by $\phi$ the corresponding row operator-valued function. Then \begin{equation}\M=\phi \h[\textbf{C}], \label{phihceq}\end{equation}
where $ \h[\textbf{C}]$ consists of $\mathbb{C}^n$-valued functions and the mapping $g \mapsto \phi g$ is an isometry from $\h[\textbf{C}]$ onto $\M$. Moreover, 
\begin{equation} \{\sum_{i=0}^n\varphi_iu_i:~u_i\in \A,~1\le i\le n\}\cap \h\label{cyclic}\end{equation}
is a dense subset of $\M$.
\end{cor}

The dimension of $\M \ominus M_z \ominus$ is the \textit{index} of the $M_z$-invariant subspace $\M$. Note that \eqref{phihceq} continues to hold when the index is infinite. In fact, this dimension can be arbitrary even in the case of weighted shifts (see \cite{esterle}). We do not know whether \eqref{cyclic} holds in the infinite index case as well. Another  natural question which arises is whether one can replace in \eqref{cyclic} the disc algebra $\A$  by the set of polynomials. We will show that this can be done in the case that $\h[\textbf{B}]$ has finite rank. For the infinite rank and index case we point to \cite{richterreinedirichlet} where sufficient conditions are given under which \eqref{cyclic} holds with $\A$ replaced by polynomials and $n = \infty$. 

 \thref{thm1} can be used to investigate reverse Carleson measures on such spaces, a concept which has been studied  in recent years in the context of $\h(b)$-spaces and model spaces (see \cite{revcarlesonross}, \cite{revcarlmodelspaces}). If $\h \cap \A$ is dense in $\h$, then a reverse Carleson measure for $\h$ is a measure $\mu$ on $\cD$ such that $\|f\|^2_\h \leq C \int_{\cD} |f|^2 d\mu$ for $f \in \h \cap \A$. In \thref{rctheorem} we prove that if the norm in $\h[\textbf{B}]$ satisfies the identity \begin{equation}
\|Lf\|^2 = \|f\|^2 - |f(0)|^2 \label{normidentity}
\end{equation} then the space cannot admit a reverse Carleson measure unless it is a backward shift invariant subspace of the Hardy space $H^2$. A class of spaces in which this identity holds has been studied in \cite{nikolskiivasyuninnotesfuncmod}, where conditions are given on $\textbf{B}$ which make the identity hold. It holds in all de Branges-Rovnyak spaces corresponding to extreme points of the unit ball of $H^\infty$, so in particular our theorem answers a question in \cite{revcarlesonross}. On the other hand if $\h$ is $M_z$-invariant, then reverse Carleson measures may exist and can be characterized in several ways. For example, in \thref{revcarl2} we show that in this case $\h = \h[\textbf{B}]$ admits a reverse Carleson measure if and only if $g := (1-\sum_{i \in I} |b_i|^2)^{-1} \in L^1(\T)$, and the measure $gd\m$ on $\T$ is essentially the minimal reverse Carleson measure for $\h$. The two conditions considered here yield almost a dichotomy. More precisely, if $\h[\textbf{B}]$ satisfies \eqref{normidentity} then it cannot be $M_z$-invariant unless it equals $H^2$. 

Another application of \thref{thm1} gives an approximation result for the orthogonal complements of $M_z$-invariant subspaces of the Bergman space $L^2_a(\mathbb{D})$. These might consist entirely of functions with bad integrability properties. For example, there are such subspaces $\M$ for which $\int_{\D} |f|^{2+\epsilon} dA = \infty$ holds for all $\epsilon > 0$ and $f \in \M \setminus \{0\}$ (see \thref{badorthocomplements}). Note that primitives of such Bergman space functions are not necessarily bounded in the disk. However, \thref{contderivdense} below shows that the set of functions in $\M$ with a primitive in $\A$ is dense in $\M$. 

The second part of the paper is devoted to the special case when $\h[\textbf{B}]$ has finite rank, according to the definition above. Intuitively speaking, in this case the structure of the backward shift $L$ resembles more a coisometry. The simplest examples are $H^2$ (rank zero) and the classical de Branges-Rovnyak spaces $\h(b)$ (rank one). Examples of higher rank $\h[\textbf{B}]$-spaces are provided by Dirichlet-type spaces $\mathcal{D}(\mu)$ corresponding to measures $\mu$ with finite support in $\cD$ (see \cite{richtersundberglocaldirichlet}, \cite{alemanhabil}).

It is not difficult to see that the rank of an $\h[\textbf{B}]$-space is unstable under with respect to equivalent Hilbert space norms. In fact, in \cite{ransforddbrdirichlet} it is shown that $\mathcal{D}(\mu)$-spaces corresponding to a measure with finite support on $\T$ admit equivalent norms under which they become a rank one space. This leads to the fundamental question whether the (finite) rank of any $\h[\textbf{B}]$-space can be reduced in this way. This question is addressed in Section \ref{isomorphismsubsec}. In \thref{modulegen} we relate the rank of $\h[\textbf{B}]$ to the number of generators of a certain $H^\infty$-submodule in the Smirnov class. In particular it turns out (\thref{nonsimilaritytheorem}) that there exist $\h[\textbf{B}]$-spaces whose rank cannot be reduced by means of any equivalent norm satisfying (A.1)-(A.3). 

In the case of finite rank $\h[\textbf{B}]$-spaces our model becomes a very powerful tool. We use it in order to establish analogues of fundamental results from the theory of $\h(b)$-spaces. Moreover, we improve \thref{cor-beurling1} to obtain a structure theorem for $M_z$-invariant subspaces which is new even in the rank one case.  More explicitly, in \thref{shiftinvcriterion} we show that a finite rank $\h[\textbf{B}]$-space is $M_z$-invariant if and only if $\log(1-\|\textbf{B}\|^2_2)$ is integrable on $\T$. If this is the case, then we show in \thref{polydense} that the polynomials are dense in the space. In \thref{Linvsubspaces} we turn to $L$-invariant subspaces and prove the analoque of a result by Sarason in  \cite{sarasondoubly}, namely that every $L$-invariant subspace of $\h[\textbf{B}]$ has the form $\h[\textbf{B}] \cap K_\theta$, where $K_\theta = H^2 \ominus \theta H^2$ is a backward shift invariant subspaces of $H^2$. Concerning the structure of $M_z$-invariant subspaces we establish the following result. 

\begin{thm} \thlabel{thm4}
	Let $\h = \h[\textbf{B}]$ be of finite rank and  $M_z$-invariant. If $\M$ is a closed $M_z$-invariant subspace of $\h$, then $\dim \M \ominus M_z\M = 1$ and any non-zero element in $\M \ominus M_z\M$ is a cyclic vector for $M_z|\M$. Moreover, if $\phi \in \M \ominus M_z \M$ is of norm $1$, then there exists a space $\h[\textbf{C}]$ invariant under $M_z$, where $\textbf{C} = (c_1, \ldots, c_k)$ and $k \leq n$, such that \begin{equation*}\M = \phi \h[\textbf{C}] \end{equation*} and the mapping $g \mapsto \phi g$ is an isometry from $\h[\textbf{C}]$ onto $\M$.
\end{thm}

\section{Basic structure} \label{prelimsec}

\subsection{Reproducing kernel.} Throughout the paper, vectors and vector-valued functions will usually be denoted by boldface letters like $\textbf{c, f}$, while operators, matrices and operator-valued functions will usually be denoted by capitalized boldface letters like \textbf{B}, \textbf{A}. The space of bounded linear operators between two Hilbert spaces $X,Y$ will be denoted by $\bhk$, and we simply write $\bx$ if $X=Y$. All appearing Hilbert spaces will be assumed separable. The identity operator on a space $X$ will be denoted by $I_X$. The backward shift operation will always be denoted by $L$, regardless of the space it acts upon, and regardless of if the operand is a scalar-valued function or a vector-valued function. The same conventions will be used for the forward shift operator $M_z$. If $Y$ is a Hilbert space, then we denote by $H^2(Y)$ the Hardy space of analytic functions $\textbf{f}: \D \to Y$ with square-summable Taylor coefficients, and $H^\infty(Y)$ is the space of bounded analytic functions from $\D$ to $Y$. The concepts of \textit{inner} and \textit{outer} functions are defined as usual (see Chapter V of \cite{nagyfoiasharmop}). The space $L^2(Y)$ will denote the space of square-integrable $Y$-valued measurable functions defined on the circle $\T$, and we will identify $H^2(Y)$ as a closed subspace of $L^2(Y)$ in the usual manner by considering the boundary values of the analytic functions $\textbf{f} \in H^2(Y)$. In the case $Y = \mathbb{C}$ we will simply write $H^2$ and $L^2$. The orthogonal complement of $H^2(Y)$ inside $L^2(Y)$ will be denoted by $\conj{H^2_0(Y)}$. The norm in $L^2(Y)$ and its subspaces will be denoted by $\| \cdot \|_2$. The inner product of two elements $f,g$ in a Hilbert space $\h$ will be denoted by $\ip{f}{g}_\h$. 

As mentioned in the introduction, we will actually work in the context of vector-valued analytic functions, which will be necessary in order to prove \thref{thm1} in full generality.  Thus, let $X, \h$ be Hilbert spaces, where $\h$ consists of analytic functions $\textbf{f}: \D \to X$. The versions of axioms (A.1)-(A.3) in the $X$-valued context are:

\begin{enumerate}[({A}.1')] \item The evaluation $\textbf{f} \mapsto \ip{\textbf{f}(\lambda)}{x}_X$ is a bounded linear functional on $\h$ for each $\lambda \in \D$ and $x \in X$,
	\item  $\hil$ is invariant under the backward shift operator $L$ given by $$L\textbf{f}(z) = \frac{\textbf{f}(z)-\textbf{f}(0)}{z},$$ and we have that $$\|L\textbf{f}\|_{\h} \leq \|\textbf{f}\|_{\h}, \quad \textbf{f} \in \hil,$$ 
	\item  the constant vectors $x \in X$ are contained in $\h$ and have the reproducing property $$\ip{\textbf{f}}{x}_\hil = \ip{\textbf{f}(0)}{x}_X \quad \mathbf{f} \in \hil, x \in X.$$ 
\end{enumerate}

By (A.1') there exists an operator-valued reproducing kernel $k_\h: \D \times \D \to \bx$ such that for each $\lambda \in \D$ and $x \in X$ the identity $$\ip{\textbf{f}}{k_\h(\cdot, \lambda)x}_\h = \ip{\textbf{f}(\lambda)}{x}_X$$ holds for $\textbf{f} \in \h$. Axiom (A.3') implies that $k_\h(z,0) = I_X$ for each $z \in \D$.

\begin{prop} \thlabel{kernel} Let $\hil$ be a Hilbert space of analytic functions which satisfies the axioms (A.1')-(A.3'). Then there exists a Hilbert space $Y$ and an analytic function $\textbf{B}: \D \to \bhk$ such that for each $z \in \D$ the operator $\textbf{B}(z): Y \to X$ is a contraction, and 
\begin{equation}k_\h(z,\lambda) = \frac{I_X-\textbf{B}(z)\textbf{B}(\lambda)^*}{1-\conj{\lambda}z}. \label{kerneleq}
\end{equation} In particular, if $X = \mathbb{C}$, then  $\textbf{B}(z)$ is an analytic   row contraction into $l^2$, i.e. there exist analytic functions $\{b_i\}_{i\ge 1}$ in $\D$  such that $$k_\h(\lambda,z) = \frac{1 - \sum_{i \ge 1}b_i(z) \conj{b_i(\lambda)}}{1-\conj{\lambda}z},\quad \sum_{i\ge 1}|b_i(z)|^2\le 1,\, z\in \D. $$ If $\dim (I_\h - LL^*)\h = n < \infty$, then there exists a representation as above, with $b_i=0,~i> n$, and such that the functions $b_1, \ldots, b_n$ are linearly independent.
\end{prop}

\begin{proof}
Let $k = k_\h$ be the reproducing kernel of $\h$. Fix a vector $x \in X$ and $\lambda \in \D$. By (A.3') we have \begin{gather*}
\ip{\textbf{f}}{(k(\cdot,\lambda)-I_X)x/\conj{\lambda}}_\h = \ip{(\textbf{f}(\lambda)-\textbf{f}(0))/\lambda}{x}_X = \ip{L\textbf{f}}{k(\cdot,\lambda)x}_\h \\ = \ip{\textbf{f}}{L^*k(\cdot, \lambda)x}_\h.
\end{gather*} It follows that $$L^*k(z, \lambda)x = \frac{(k(z,\lambda) - I_X)x}{\conj{\lambda}}$$ and \begin{equation}(I_\h - LL^*)k(z,\lambda)x = k(z, \lambda)x - \frac{(k(z,\lambda) - I_X)x}{\conj{\lambda}z}. \label{eq32} \end{equation} By (A.2') the operator $P := I_\h - LL^*$ is positive. Therefore $Pk(\cdot, \lambda)$ is a positive-definite kernel and hence it has a factorization $Pk(z,\lambda) = \tilde{\textbf{B}}(z)\tilde{\textbf{B}}(\lambda)^*$ for some operator-valued analytic function $\tilde{\textbf{B}}: \D \to \bhk$ (see Chapter 2 of \cite{mccarthypick}). We can now solve for $k$ in \eqref{eq32} to obtain $$k(z,\lambda) = \frac{I_X - \textbf{B}(z)\textbf{B}(\lambda)^*}{1-\conj{\lambda}z},$$ where we have set $\textbf{B}(z) := z\tilde{\textbf{B}}(z)$. It is clear from this expression and the positive-definiteness of $k$ that $\textbf{B}(z)$ must be a contraction for every $z \in \D$. If $\dim(I_\h - LL^*) < \infty$, then the last assertion of the proposition follows in a standard manner from \eqref{eq32} and the spectral theorem applied to the finite rank operator $I_\h - LL^*$.
\end{proof}

%
%

\subsection{Model.} The space $\h$ is completely determined by the function $\textbf{B}: \D \to \bhk$ appearing in \thref{kernel}. To emphasize this we will on occasion write $\h[\textbf{B}]$ in place of $\h$. The function $\textbf{B}$ admits a non-tangential boundary value $\textbf{B}(\zeta)$ for almost every $\zeta \in \T$ (convergence in the sense of strong operator topology), and the operator 
\begin{equation}\Delta(\zeta) = (I_{\h} - \textbf{B}(\zeta)^*\textbf{B}(\zeta))^{1/2} \label{delta}\end{equation} induces in a natural way a multiplication operator from $H^2(Y)$ to $L^2(Y)$. The space $\deh$  is a subspace of $L^2(Y)$ which is invariant under the operator $M_\z$ given by $M_\z\textbf{g}(\z) = \z\textbf{g}(\z), \z \in \T.$ This implies that it can be decomposed as \begin{equation}
\deh = W \oplus \Theta H^2(Y_1) \label{decomp}
\end{equation} where $M_\z$ acts unitarily on $W$, $Y_1$ is an auxilliary Hilbert space, $\Theta: \T \to \mathcal{B}(Y_1, Y)$ is a measurable operator-valued function such that for almost every $\z \in \T$ the operator $\Theta(\z): Y_1 \to Y$ is isometric, and the functions in $W$ and $\Theta H^2(Y_1)$ are pointwise orthogonal almost everywhere, i.e, $\ip{\textbf{f}(\z)}{ \textbf{g}(\z)}_Y = 0$ for almost every $\z \in \T$ for any pair $\textbf{f} \in W$ and $\textbf{g} \in \Theta H^2(Y_1)$ (see Theorem 9 of Lecture VI in \cite{helsonbook}). It follows that $\Theta(\z)^* \textbf{f}(\z) = 0$ for all $\textbf{f} \in W$, and consequently $$\Theta^*\deh = \Theta^*W \oplus \Theta^* \Theta H^2(Y_1) = \{0\} \oplus H^2(Y_1).$$ The above computation shows that $\Theta^*\Delta$ maps $H^2(Y)$ to a dense subset of $H^2(Y_1)$, and so standard theory of operator-valued functions implies that there exists an analytic outer function $\textbf{A}: \D \to \mathcal{B}(Y, Y_1)$ such that $\textbf{A}(\z) = \Theta^*(\z)\Delta(\z)$ for almost every $\z \in \T$. 

\begin{thm} \thlabel{modeltheorem}
There exists an isometric embedding $J: \h[\textbf{B}] \to H^2(X) \oplus \deh$ satisfying the following properties.

\begin{enumerate}[(i)]
\item A function $\textbf{f} \in H^2(X)$ is contained in $\h[\textbf{B}]$ if and only if there exists $\textbf{g} \in \deh$ such that \[\textbf{B}^*\textbf{f} + \Delta \textbf{g} \in \conj{H^2_0(Y)}.\] If this is the case, then $\textbf{g}$ is unique and \[J\textbf{f} = (\textbf{f}, \textbf{g}).\]

\item If $J\textbf{f} = (\textbf{f}, \textbf{g})$ and $\textbf{g} = \textbf{w} + \Theta \textbf{f}_1$ is the decomposition of $\textbf{g}$ with respect to \eqref{decomp}, then $$JL\textbf{f} = (L\textbf{f}, \, \conj{\z}\textbf{w} + \Theta L\textbf{f}_1).$$

\item The orthogonal complement of $J\h[\textbf{B}]$ inside $H^2(X) \oplus \deh$ is \[ (J\h[\textbf{B}])^\perp = \{(\textbf{B}\mathbf{h}, \Delta \mathbf{h}) : \mathbf{h} \in H^2(Y)\}.\] 
\end{enumerate}
\end{thm}

\begin{proof} Let $K = H^2(X) \oplus \deh, U = \{(\textbf{B}\mathbf{h}, \Delta \mathbf{h}) : \mathbf{h} \in H^2(Y)\}$ and note that $U$ is a closed subspace of $K$. Let $P$ be the projection taking a tuple $(\textbf{f},\textbf{g}) \in K$ to $\textbf{f} \in H^2(X)$. Then $P$ is injective on $K \ominus U$. Indeed, if $(0,\textbf{g})$ is orthogonal to $U$, then $\textbf{g} \in \text{clos}(\Delta H^2(Y))$ is orthogonal to $\Delta H^2(Y)$, and hence $\textbf{g} = 0$. Any analytic function $\textbf{f}$ can thus appear in at most one tuple $(\textbf{f},\textbf{g}) \in K \ominus U$. We define $\h_0 = P(K \ominus U)$ as the space of analytic $X$-valued functions with the norm $$\|\textbf{f}\|_{\h_0}^2 := \|\textbf{f}\|_2^2 + \|\textbf{g}\|^2_2.$$ We will see that $\h[\textbf{B}] = \h_0$ by showing that the reproducing kernels of the two spaces are equal. Note that for each $c \in X$ the tuple
\begin{equation*}
	k_{c,\lambda} = \Bigg( \frac{(I_{X} - \textbf{B}(z)\textbf{B}(\lambda)^*)c}{1-\conj{\lambda}z}, \frac{-\Delta(\z)\textbf{B}(\lambda)^*c}{1-\conj{\lambda}z}\Bigg) \in K
\end{equation*} is orthogonal to $U$, and therefore its first component $Pk_{c,\lambda}$ defines an element of $\h_0$. Moreover, it follows readily from our definitions that for any $\textbf{f} \in \hil_0$ we have \begin{eqnarray*}
\ip{\textbf{f}}{Pk_{c,\lambda}}_{\h_0} = \ip{\textbf{f}(\lambda)}{c}_{X}.
\end{eqnarray*} Thus the reproducing kernel of $\h_0$ equals the one given by \eqref{kerneleq}, and so $\h[\textbf{B}] = \h_0$.

It is clear from the above paragraph that a function $\textbf{f} \in H^2(X)$ is contained in $\h[\textbf{B}]$ if and only if there exists $\textbf{g} \in \deh$ such that $(\textbf{f}, \textbf{g})$ is orthogonal to $(\textbf{B}\textbf{h}, \Delta \textbf{h})$ for all $\textbf{h} \in H^2(Y)$, that is to say, if and only if $\textbf{B}(\zeta)^*\textbf{f}(\z) + \Delta(\z) \textbf{g}(\z) \in \conj{H^2_0(Y)}$. If we let $J = P^{-1}$, then $J\textbf{f} = (\textbf{f}, \textbf{g})$, and part (i) follows. Part (iii) holds by construction. In order to prove (ii), it will be sufficient by (i) to show that \begin{equation}
\textbf{B}^*L\textbf{f} + \Delta(\conj{\z}\textbf{w} + \Theta L \textbf{f}_1) \in \conj{H^2_0(Y)}. \label{eq39} \end{equation}
Let $\textbf{A} = \Theta^*\Delta$ be the analytic function mentioned above. We have that \begin{equation}\conj{\z}(\textbf{B}^*\textbf{f} + \Delta \textbf{g}) = \conj{\z} \textbf{B}^*\textbf{f} + \Delta \conj{\z} \textbf{w} + \conj{\z} \textbf{A}^*\textbf{f}_1 \in \conj{H^2_0(Y)}.\label{eq19} \end{equation} The term $\textbf{A}^* L \textbf{f}_1$ differs from $\conj{\z}\textbf{A}^*\textbf{f}_1$ only by a function in $\conj{H^2_0(Y)}$, and the same is true for $\textbf{B}^*L\textbf{f}$ and $\conj{\z}\textbf{B}^*\textbf{f}$. Thus \eqref{eq39} follows from \eqref{eq19}.
\end{proof}

\subsection{Analytic model.} The model of \thref{modeltheorem} can be greatly simplifed if $W = \{0\}$ in the decomposition \eqref{decomp}. The condition for when this occurs can be expressed in terms of $L$.

\begin{cor} \thlabel{nowiener} We have $W = \{0\}$ in \eqref{decomp} if and only if $\|L^n\textbf{f}\|_{\h[\textbf{B}]} \to 0$ as $n \rightarrow \infty$, for all $f \in \h[\textbf{B}]$. 
\end{cor}

\begin{proof}
Assume first that $W = \{0\}$. If $J\textbf{f} = (\textbf{f}, \Theta \textbf{h})$, then by (ii) of \thref{modeltheorem} we have that $$\|L^n \textbf{f}\|^2_{\h[\textbf{B}]} = \|L^n\textbf{f}\|^2_2 + \|L^n \textbf{h}\|^2_2$$ which clearly tends to 0 as $n \to \infty$. 

Conversely, note that the convergence of $L^n\textbf{f}$ to zero implies that if $J\textbf{f} = (\textbf{f}, \textbf{g})$ and $\textbf{g} = \textbf{w} + \Theta \textbf{h}$ is the decomposition of $\textbf{g}$ with respect to \eqref{decomp}, then $\textbf{w} = 0$ (else $\lim_{n \to \infty} \|L^n\textbf{f}\|_{\h[\textbf{B}]} \geq \|\textbf{w}\|_2$ by (ii) of \thref{modeltheorem}). Thus for any $\textbf{w} \in W$ we have $J\textbf{f} \perp \textbf{w}$ for all $\textbf{f} \in \h[\textbf{B}]$, and so $(0,\textbf{w}) = (\textbf{B}\textbf{h}, \Delta \textbf{h})$ for some $\textbf{h} \in H^2(Y)$ by \thref{modeltheorem}. Since $\textbf{B}\textbf{h} = 0$, we deduce that $\textbf{w} = \Delta \textbf{h} = \textbf{h}$, for example by approximating the function $x \mapsto \sqrt{1-x}$ uniformly on $[0,1]$ by a sequence of polynomials $p_n$ with $p_n(0) = 1$, so that $\Delta \textbf{h} = \lim_{n \to \infty} p_n(\textbf{B}^*\textbf{B})\textbf{h} = \textbf{h}$. It follows that $\textbf{w} = \textbf{h} \in H^2(Y)$. Since $\textbf{w}$ was arbitrary, we deduce that $W \subset H^2(Y)$, and since $M_\z$ acts unitarily on $W$, we must have $W = \{0\}$. 
\end{proof}

Assume we are in the case described by \thref{nowiener}. Then \eqref{decomp} reduces to $$\deh = \Theta H^2(Y_1).$$ Thus it holds that $\conj{\Im \Delta (\z)} = \Im \Theta(\z)$ for almost every $\z \in \T$. Since $\Theta(\z) \Theta(\z)^*$ is equal almost everywhere to the projection of $Y$ onto $\Im \Theta(\z)$, we obtain $$\textbf{A}^*(\z)\textbf{A}(\z) = \Delta(\z)\Theta(\z)\Theta(\z)^*\Delta(\z) = \Delta(\z)^2$$ and consequently $\textbf{B}(\z)^*\textbf{B}(\z) + \textbf{A}(\z)^*\textbf{A}(\z) = 1_Y$ for almost every $\z \in \T$. The following theorem is a reformulation of \thref{modeltheorem}. We omit the proof, which can be easily deduced from the proof of \thref{modeltheorem} with $\textbf{A}$ playing the role of $\Delta$.

\begin{thm} \thlabel{modeltheoremanalytic} Let $\h[\textbf{B}]$ be such that $\|L^n\textbf{f}\|_{\h[\textbf{B}]} \to 0$ as $n \rightarrow \infty$, for all $f \in \h[\textbf{B}]$. There exists an auxilliary Hilbert space $Y_1$, an outer function $\textbf{A}: \D \to \mathcal{B}(Y, Y_1)$ such that $$\textbf{B}(\z)^*\textbf{B}(\z) + \textbf{A}(\z)^*\textbf{A}(\z) = 1_Y$$ for almost every $\z \in \T$, and an isometric embedding $J:\h[\textbf{B}] \to H^2(X) \oplus H^2(Y_1)$ satisfying the following properties.

\begin{enumerate}[(i)]
\item A function $\textbf{f} \in H^2(X)$ is contained in $\h[\textbf{B}]$ if and only if there exists $\textbf{f}_1 \in H^2(Y_1)$ such that $$\textbf{B}^*\textbf{f} + \textbf{A}^* \textbf{f}_1 \in \conj{H^2_0(Y)}.$$ If this is the case, then $\textbf{f}_1$ is unique, and \[J\textbf{f} = (\textbf{f}, \textbf{f}_1).\]
\item If $J\textbf{f} = (\textbf{f}, \textbf{f}_1)$, then\[J\textbf{f} = (L\textbf{f}, L \textbf{f}_1).\] 

\item The orthogonal complement of $J\h[\textbf{B}]$ inside $H^2(X) \oplus H^2(Y_1)$ is \[(\h[\textbf{B}])^\perp = \{(\textbf{B}\mathbf{h}, \textbf{A}\mathbf{h}) : \mathbf{h} \in H^2(Y)\}.\] 

\end{enumerate} 
\end{thm}	

\subsection{A formula for the norm.} \label{formulasubsection} As pointed out in the introduction, the model of \thref{modeltheorem} has a connection to a useful formula for the norm in Hilbert spaces of analytic functions.

\begin{prop} \thlabel{formula}

Let $\h$ be a Hilbert space of analytic functions which satisfies (A.1')-(A.3'). For $\textbf{f} \in \h$ we have \begin{equation}
\|\textbf{f}\|^2_{\h} = \|\textbf{f}\|^2_2 + \lim_{r \to 1} \int_\T \Big\|z\frac{\textbf{f}(z)-\textbf{f}(r\lambda)}{z-r\lambda}\Big\|^2_{\h} - r^2\Big\|\frac{\textbf{f}(z)-\textbf{f}(r\lambda)}{z-r\lambda}\Big\|^2_{\h} d\m(\lambda). \label{eq57}
\end{equation}
\end{prop} 
\noindent Note that the formula makes sense even if $\h$ is not invariant for $M_z$, since for $\lambda \in \D$ we have $$z\frac{\textbf{f}(z)-\textbf{f}(\lambda)}{z-\lambda} = \frac{z\textbf{f}(z)- \lambda\textbf{f}(\lambda)}{z-\lambda} - \textbf{f}(\lambda) = (1-\lambda L)^{-1}\textbf{f}(z) - \textbf{f}(\lambda) \in \h,$$ and $(1-\lambda L)^{-1}$ exists since $L$ is a contraction on $\h$. Versions of the above formula have been used in a crucial way in several works related to the structure of invariant subspaces, see for example \cite{afr}, \cite{alemanrichtersimplyinvariant} and \cite{ars}. We shall prove the formula by verifying that if $J\textbf{f} = (\textbf{f}, \textbf{g})$ in \thref{modeltheorem}, then the limit in \eqref{eq57} is equal to $\|\textbf{g}\|_2^2$. Actually, we will prove a stronger result than \thref{formula}, one which we will find useful at a later stage. In the next theorem and in the sequel we use the notation $$\L_\lambda := L(1-\lambda L)^{-1}, \lambda \in \D.$$ One readily verifies that $$L_\lambda \textbf{f}(z) = \frac{\textbf{f}(z)-\textbf{f}(\lambda)}{z - \lambda}.$$ 

\begin{thm} \thlabel{modelformulaconnection} Let $J\textbf{f} = (\textbf{f}, \textbf{w} + \Theta \textbf{f}_1)$ as in \thref{modeltheorem}, or $J\textbf{f} = (\textbf{f}, \textbf{f}_1)$ as in \thref{modeltheoremanalytic}. In both cases, we have that 
\begin{enumerate}[(i)]
\item $\|zL_{\lambda}\textbf{f}\|^2_{\h[\textbf{B}]} - \|L_{\lambda}\textbf{f}\|^2_{\h[\textbf{B}]} = \|\textbf{f}_1(\lambda)\|^2_{Y_1} $,
\item $\|\textbf{f}_1\|_2^2 = \lim_{r \to 1} \int_\T \|zL_{r\lambda}\textbf{f}\|^2_{\h[\textbf{B}]} - \|L_{r\lambda}\textbf{f}\|^2_{\h[\textbf{B}]} d\m(\lambda)$,
\item $\|\textbf{w}\|_2^2 = \lim_{r \to 1} \int_\T (1-r^2)\|L_{r\lambda}\textbf{f}\|^2_{\h[\textbf{B}]} d\m(\lambda)$. 
\end{enumerate}
In particular, \eqref{eq57} holds.
\end{thm}

\begin{proof} We shall only carry out the proof in the context of \thref{modeltheorem}, the other case being similar. First, we claim that for $\lambda \in \D$ we have that \begin{equation} JzL_\lambda \textbf{f} = \big(zL_\lambda \textbf{f}, (1-\lambda \conj{\z})^{-1}\textbf{w} + \Theta (zL_\lambda \textbf{f}_1 + \textbf{f}_1(\lambda))\big). \label{eq114} \end{equation} This follows from part (i) of \thref{modeltheorem}. Indeed, we have to check that	
\begin{gather*}
\textbf{B}(\z)^*\z \frac{\textbf{f}(\z)- \textbf{f}(\lambda)}{\z - \lambda} + \Delta (\z) \frac{\textbf{w}(\z)}{1-\lambda \conj{\z}} + \Delta(\z)\Theta(\z)\z \frac{\textbf{f}_1(\z) - \textbf{f}_1(\lambda)}{\z - \lambda} + \Delta(\z)\Theta (\z) \textbf{f}_1(\lambda) \\ = \frac{\textbf{B}(\z)^* \textbf{f}(\z) + \Delta(\z)\textbf{g}(\z)}{1-\lambda \conj{\z}} - \frac{\textbf{B}(\z)^*\textbf{f}(\lambda)}{1-\lambda \conj{\z}} - \Big( \frac{\textbf{A}(\z)^*\textbf{h}(\lambda)}{1-\lambda\conj{\z}} - \textbf{A}(\z)^*\textbf{f}_1(\lambda)\Big)\end{gather*}
lies in $\conj{H^2_0(Y)}$, and this is true since each of the three terms in the last line lies in $\conj{H^2_0(Y)}$. Similarly, we have $$JL_\lambda \textbf{f} = \big(L_\lambda \textbf{f}, \conj{\z}(1-\lambda \conj{\z})^{-1}\textbf{w} + \Theta L_\lambda \textbf{f}_1\big).$$ Actually, this can be seen immediately by applying (ii) of \thref{modeltheorem} to \eqref{eq114}. Since $J$ is an isometry we have that \begin{equation}\|zL_\lambda \textbf{f}\|^2_{\h[\textbf{B}]} = \|zL_\lambda \textbf{f}\|^2_2 + \|(1-\lambda \conj{\z})^{-1}\textbf{w}\|^2_2 + \|zL_\lambda \textbf{f}_1 + \textbf{f}_1(\lambda)\|^2_2 \label{eq111} \end{equation} and \begin{equation}\|L_\lambda \textbf{f}\|^2_{\h[\textbf{B}]} = \|L_\lambda \textbf{f}\|^2_2 + \|\conj{\z}(1-\lambda \conj{\z})^{-1}\textbf{w}\|^2_2 + \|L_\lambda \textbf{f}_1\|^2_2. \label{eq112} \end{equation} The difference of \eqref{eq111} and \eqref{eq112} equals $\|\textbf{f}_1(\lambda)\|^2_{Y_1}$ which gives (i). Part (ii) is immediate from (i). We will deduce part (iii) from \eqref{eq112}. A brief computation involving power series shows that if $T$ is a contraction on a Hilbert space $\h$, then we have $\|T^nx\|_\h \to 0$ if and only if $$\lim_{r\to 1^-} \int_\T (1-r^2)\|T(1-r\lambda T)^{-1}x\|^2_{\h} d\m(\lambda) = 0.$$ We apply this to the $T = L$ acting on the Hardy spaces, and deduce (iii) from \eqref{eq112}: \begin{gather*}\lim_{r \to 1} \int_\T (1-r^2)\|L_{r\lambda}\textbf{f}\|^2_{\h[\textbf{B}]} d\m(\lambda) = \lim_{r \to 1} \int_\T (1-r^2) \|(1-r\lambda \conj{\z})^{-1} \textbf{w}\|^2_2 d\m(\lambda) \\ = \lim_{r \to 1} \int_\T \Bigg(\int_\T (1-r^2)|1-r\lambda \conj{\z}|^{-2} d\m(\lambda) \Bigg)\|\textbf{w}(\z)\|^2_Y d\m(\z) = \|\textbf{w}\|^2_2. \end{gather*}
\end{proof}

\subsection{Examples.} \label{examplessubsec} We end this section by discussing some examples of spaces which satisfy our assumptions, some of which were already mentioned in the introduction. 

\subsubsection{De Branges-Rovnyak spaces.} If $\textbf{B} = b$ is a non-zero scalar-valued function, then $\Delta = (1-|b|^2)^{1/2}$ is an operator on a $1$-dimensional space, and $\text{clos}(\Delta H^2) \subseteq L^2$ is either of the form $\theta H^2$ for some unimodular function $\theta$ on $\T$, or it is of the form $L^2(E) := \{ f \in L^2 : f \equiv 0 \text{ a.e. on } \T \setminus E \}$. The first case corresponds to $b$ which are non-extreme points of the unit ball of $H^\infty$, while the second case corresponds to the extreme points. It is in the first case that the model of \thref{modeltheoremanalytic} applies to $\h(b)$. 

\subsubsection{Weighted $H^2$-spaces.} Let $w = (w_n)_{n \geq 0}$ be a sequence of positive numbers and $H^2_w$ be the space of analytic functions in $\D$ which satisfy \[ \|f\|^2_{H^2_w} := \sum_{n=0}^\infty w_n|f_n|^2 < \infty,\] where $f_n$ is the $n$th Taylor coefficient of $f$ at $z = 0$. If $w_0 = 1$ and $w_{n+1} \geq w_n$ for $n \geq 0$, then $H^2_w$ satisfies (A.1)-(A.3). The reproducing kernel of $H^2_w$ is given by \[k(z,\lambda) = \sum_{n = 0}^\infty \frac{\conj{\lambda^n}z^n}{w_n} = \frac{1 - \sum_{n=1}^\infty (1/w_{n-1} - 1/w_n)\conj{\lambda^n}z^n}{1-\conj{\lambda}z}, \] and thus $H^2_w = \h[\textbf{B}]$ with $\textbf{B}(z) = \big( \sqrt{1/w_{n-1} - 1/w_n}z^n \big)_{n=1}^\infty$. 

\subsubsection{Dirichlet-type spaces.} Let $\mu$ be a positive finite Borel measure on $\cD$. The space $\mathcal{D}(\mu)$ consists of analytic functions which satisfy \[ \|f\|^2_{\mathcal{D}(\mu)} := \|f\|^2_2 + \int_{\T}\int_{\cD} \frac{|f(z) - f(\z)|^2}{|z - \z|^2} d\mu(z) d\m(\z).\] The choice of $d\mu = d\m$ produces the classical Dirichlet space. It is easy to verify directly from this expression that $\mathcal{D}(\mu)$ satisfies axioms (A.1)-(A.3), but it is in general difficult to find an expression for $\textbf{B}$ corresponding to the space as in \thref{kernel}. In the special case that $\mu = \sum_{i=1}^n c_i\delta_{z_i}$ is a positive sum of unit masses $\delta_{z_i}$ at distinct points $z_i \in \cD$ the space $\mathcal{D}(\mu)$ is an $\h[\textbf{B}]$-space of rank $n$. An isometric embedding $J: \mathcal{D}(\mu) \to (H^2)^{n+1}$ satisfying the properties listed in \thref{modeltheoremanalytic} is given by \[f(z) \mapsto \Big(f(z), \frac{f(z)-f(z_1)}{z - z_1}, \ldots, \frac{f(z)-f(z_n)}{z-z_n} \Big). \]

\subsubsection{Cauchy duals.} Let $\h$ be a Hilbert space of analytic functions which contains all functions holomorphic in a neighbourhood of $\cD$, on which the forward shift operator $M_z$ acts as a contraction and such that $\ip{f}{1}_\h = f(0)$ holds for $f \in \h$. Consider the function \[ Uf(\lambda) = \ip{(1-\lambda z)^{-1}}{f(z)}_\h = \sum_{n=0}^\infty \lambda^n \ip{z^n}{f(z)}_{\h}.\] Then $Uf$ is an analytic function of $\lambda$, and $UM_z^*f = LUf$. If $\h^*$ is defined to be the space of functions of the form $Uf$ for $f \in \h$, with the norm $\|Uf\|_{\h^*} = \|f\|_{\h}$, then it is easy to verify that $\h^*$ is a Hilbert space of analytic functions which satisfies (A.1)-(A.3). The space $\h^*$ is the so-called Cauchy dual of $\h$ (see \cite{acppredualsqp}).

\section{Density of functions with continuous extensions to the closed disk} \label{densitysection}
The goal of this section is to prove \thref{conttheorem} and \thref{conttheoremcor}, which are vector-valued generalizations of \thref{thm1} mentioned in the introduction. We will first recall a few facts about the disk algebra $\A$ and the vector-valued Smirnov classes $N^+(Y)$. 

\subsection{Disk algebra, Cauchy transforms and the Smirnov class.} \label{dasubsec}
Let $\A$ denote the disk algebra, the space of scalar-valued analytic functions defined in $\D$ which admit continuous extensions to $\cD$. It is a Banach space if given the norm $\|f\|_\infty = \sup_{z \in \D} |f(z)|$, and the dual of $\A$ can be identified with the space $\C$ of Cauchy transforms of finite Borel measures $\mu$ supported on the circle $\T$. A Cauchy transform $f$ is an analytic function in $\D$ which is of the form $f(z) = C\mu(z) := \int_\T \frac{1}{1-z\conj{\z}} d\mu(\z)$ for some Borel measure $\mu$. The duality between $\A$ and $\C$ is realized by $$\ip{h}{f} = \lim_{r \rightarrow 1^-}\int_\T h(r\z)\conj{f(r\z)} d\m(\zeta) = \int_{\T} h \conj{d\mu}, \quad h \in \A, f = C\mu$$ and the norm $\|f\|_\mathcal{C}$ of $f$ as a functional on $\A$ is given by $\|f\|_\mathcal{C} = \inf_{\mu : C\mu = f} \|\mu\|,$ where $\|\mu\|$ is the total variation of the measure $\mu$. The space $\C$ is continuously embedded in the Hardy space $H^p$ for each $p \in (0,1)$. More precisely, we have for each fixed $p \in (0,1)$ the estimate $\|f\|_p = (\int_\T |f|^p d\m)^{1/p} \leq c_p \|f\|_{\C}.$ As a dual space of $\A$, the space $\C$ can be equipped with the weak-star topology, and a sequence $(f_n)$ converges weak-star to $f \in \C$ if and only if $\sup_n \|f_n\|_{\C} < \infty$ and $f_n(z) \to f(z)$ for each $z \in \D$. See \cite{cauchytransform} for more details.

If $Y$ is a Hilbert space, then the Smirnov class $N^+(Y)$ consists of the functions $\textbf{f}: \D \to Y$ which can be written as $\textbf{f} = \textbf{u}/v$, where $\textbf{u} \in H^\infty({Y})$ and $v: \D \to \mathbb{C}$ is a bounded outer function. In the case $Y = \mathbb{C}$ we will simply write $N^+$. The class $N^+(Y)$ satisfies the following \textit{Smirnov maximum principle}: if $\textbf{f} \in N^+(Y)$, then we have that $\int_\T \|\mathbf{f}(\z)\|_{Y}^2 d\m(\z) < \infty$ if and only if $\mathbf{f} \in H^2(Y)$ (see Theorem A in Section 4.7 of \cite{rosenblumrovnyakhardyclasses}).

\subsection{Proof of the density theorem.} \label{proofsubsec}

	The proof will depend on a series of lemmas. The first two are routine exercises in functional analysis and the proofs of those will be omitted. 

\begin{lemma}
	\thlabel{weakstarapprox}
	Let $B$ be a Banach space, $B'$ be its dual space, and $S \subset B'$ be a linear manifold. If $l \in B'$ annihilates the subspace $\cap_{s \in S} \ker s \subset B$, then $l$ lies in the weak-star closure of $S$.
\end{lemma}


\begin{lemma} \thlabel{convlemma}
	Let $\{h_j\}$ be a sequence of scalar-valued analytic functions in $\D$, with $\sup_n \|h_n\|_\infty < \infty$, and which converges uniformly on compacts to the function $h$. If the sequence $\{\textbf{g}_j\}_{j=1}^\infty$ of functions in $L^2(Y)$ converges in norm to $\textbf{g}$, then $h_j\textbf{g}_j$ converges weakly in $L^2(Y)$ to $h\textbf{g}$. 
\end{lemma}


The next two lemmas are more involved. Let $\A^n = \A \times \ldots \times \A$ denote the product of $n$ copies of the disk algebra. The dual of $\A^n$ can then be identified with $\C^n$, the space of $n$-tuples of Cauchy transforms $\textbf{f} = (f^1, \ldots, f^n)$, normed by $\|\textbf{f}\|_{\C^n} = \sum_{i=1}^n \|f^i\|_\C.$ The main technical argument needed for the proof of \thref{conttheorem} is contained in the following lemma.

\begin{lemma} \thlabel{multiplyintoh2lemma}
	Let $\{\textbf{f}_m\}_{m=1}^\infty$ be a sequence in $\C^n$ which converges weak-star to $\textbf{f}$. There exists a subsequence $\{\textbf{f}_{m_k}\}_{k=1}^\infty$ and a sequence of outer functions $M_k: \D \to \mathbb{C}$ satisfying the following properties:
	
	\begin{enumerate}[(i)]
	\item $\|M_k\|_\infty \leq 1$,
	\item $M_k$ converges uniformly on compacts to a non-zero outer function $M$,
	\item $M_k\textbf{f}_{m_k} = (M_kf^1_{m_k}, M_kf^2_{m_k}, \ldots, M_kf^n_{m_k}) \in (H^2)^n$,
	\item the sequence $\{ M_k\mathbf{f_{m_k}}\}_{m=1}^\infty$ converges weakly to $M\mathbf{f}$ in $(H^2)^n$.
	\end{enumerate}
\end{lemma} 

\begin{proof} Let $\textbf{g} = (g^1, g^2, \ldots, g^n)$ be an $n$-tuple of Cauchy transforms and for some fixed choice of $p \in (1/2, 1)$ let $$s(\z) = \max(\sum_{i=1}^n |g^i(\zeta)|^p,1 ), \quad \z \in \T.$$ Then $s$ is integrable on the circle and $\int_{\T} s \,d\m  \leq C_1\|\textbf{g}\|_{\C^n}^p,$ where the constant $C_1 > 0$ depends on $p$ and $n$, but is independent of $\textbf{g} \in \mathcal{C}^n$. We let $H$ be the Herglotz transform of $s$, that is $$H(z) = \int_{\T} \frac{\z + z}{\z - z} s(\z) d\m(\z).$$ Note that the real part of $H$ is the Poisson extension of $s(\z)$ to $\D$. This shows that $H$ has positive real part (hence is outer), $|H| \geq s \geq 1$ on $\T$ and $|H(z)| \geq 1$ for all $z \in \D$. We also have that $H(0) = \int_\T s \, d\m \leq C_1\|\textbf{g}\|^p_{\C^n}.$ Let $q = 2 - 2p \in (0,1).$ For each $i \in \{1, \ldots, n\}$ we have the estimate $$\int_\T |g^i/H|^2 d\m \leq \int_\T |g^i/s|^2 d\m \leq \int_\T |g^i|^{2-2p} d\m = \|g^i\|_{q}^{q} \leq C_2 \|g^i\|^q_{\C}.$$ Since the functions $g^i$ are in $\C \subset N^+$, the Smirnov maximum principle implies that $g^i/H \in H^2$, or equivalently $\textbf{g}/H \in (H^2)^n$. Moreover, $\|\textbf{g}/H\|_{(H^2)^n} \leq C\|\textbf{g}\|^q_{\C^n}$, with constant $C > 0$ depending only on the fixed choice of $p$ and the dimension $n$. 

Let now $\{\textbf{f}_m\}_{m=1}^\infty$ be a sequence in $\C^n$ which converges weak-star to $\textbf{f}$, meaning that $\textbf{f}_m$ converges pointwise to $\textbf{f}$ in $\D$, and we have $\sup_{m} \|\textbf{f}_m\|_{\C^n} < \infty$. For each integer $m \geq 1$ we construct the function $H = H_m$ as above. By what we have established above, $\{\textbf{f}_m/H_m\}_{m=1}^\infty$ is a bounded sequence in $(H^2)^n$. Since $\|1/H_m\|_{\infty} \leq 1$ and $H_m(0) \leq C_1\|\textbf{f}_m\|^p_{\C^n}$, there exists a subsequence $\{m_k\}_{k=1}^\infty$ such that $M_k = 1/{H_{m_k}}$ converges uniformly on compacts to a non-zero analytic function $M$. Then $M$ has positive real part, since each of the functions $M_k$ has positive real part, and therefore $M$ is outer. The sequence $\{M_k\textbf{f}_{m_k}\}_{k=1}^\infty$ is bounded in $(H^2)^n$ and converges pointwise to the function $M\textbf{f}$, which is equivalent to weak convergence in $(H^2)^n$.	
\end{proof}

For the rest of the section, let $\h = \h[\textbf{B}]$ be a fixed space of $X$-valued functions, where $X$ is a finite dimensional Hilbert space. Thus $\textbf{B}$ takes values in $\bhk$ for some auxilliary Hilbert space $Y$. As before, set $\Delta(\z) = (I_Y - \textbf{B}(\z)^*\textbf{B}(\z))^{1/2}$ for $\z \in \T$. Fix an orthonormal basis $\{e_i\}_{i=1}^n$ for the finite dimensional Hilbert space $X$. We can define a map from $H^2(X)$ to $(H^2)^n$ by the formula $\textbf{f} \mapsto (f^i)_{i=1}^n$ where the components $f^i$ are the coordinate functions $f^i(z) = \ip{\textbf{f}(z)}{e_i}_X.$ Then a function $\textbf{f} \in H^2(X)$ has a continuous extension to $\cD$ if and only if all of its coordinate functions $f^i$ are contained in the disk algebra $\A$.

\begin{lemma} \thlabel{Sset}
	\thlabel{techlemma} Let $S \subset \mathcal{C}^n \oplus \deh$ be the linear manifold consisting of tuples of the form $$(\textbf{B} \textbf{f}, \Delta \textbf{f})$$ for some $\textbf{f} \in N^+(Y)$. Then $S$ is weak-star closed in $\mathcal{C}^n \oplus \deh$. 
\end{lemma}

\begin{proof}
Since $\A^n \oplus \deh$ is separable, Krein-Smulian theorem implies that it is enough to check weak-star sequential closedness of the set $S$. Thus, let the sequence \[\{(\textbf{B} \textbf{f}_m, \Delta \textbf{f}_m)\}_{m=1}^\infty = \{(\textbf{h}_m, \textbf{g}_m)\}_{m=1}^\infty \subset \mathcal{C}^n \oplus \deh\] converge in the weak-star topology to $(\textbf{h}, \textbf{g}).$ Then $\{ \textbf{g}_m \}_{m=1}^\infty$ converges weakly in the Hilbert space $\deh$, and by passing to a subsequence and next to the Ces\`aro means of that subsequence, we can assume that the sequence $\{\textbf{g}_m\}_{m=1}^\infty$ converges to $\textbf{g}$ in the norm. By applying \thref{multiplyintoh2lemma} and \thref{convlemma} we obtain a sequence of outer functions $\{M_k\}_{k=1}^\infty$ and an outer function $M$ such that $\{(M_k\textbf{h}_{m_k}, M_k\textbf{g}_{m_k})\}_{k=1}^\infty$ converges weakly in the Hilbert space $H^2(X) \oplus \deh$ to $(M\textbf{h}, M\textbf{g})$. Note that $$(M_k\textbf{h}_{m_k}, M_k\textbf{g}_{m_k}) = (\textbf{B} M_k \textbf{f}_{m_k}, \Delta M_k \textbf{f}_{m_k}),$$ and $$\int_\T \|M_k\textbf{f}_{m_k}\|^2_Y d\m = \int_\T \|\textbf{B} M_k\textbf{f}_{m_k}\|^2_X d\m + \int_\T \|\Delta M_k\textbf{f}_{m_k}\|^2_Y d\m < \infty.$$ Since $M_k\textbf{f}_{m_k}$ is in $N^+(Y)$, the Smirnov maximum principle implies that we have $M_k\textbf{f}_{m_k} \in H^2(Y)$, and consequently the tuples $(\textbf{B}M_k\textbf{f}_{m_k}, \Delta M_k\textbf{f}_{m_k})$ are contained in the closed subspace $U = \{ (\textbf{B}\textbf{h}, \Delta\textbf{h}) : \textbf{h} \in H^2(Y)\}$. It follows that the weak limit $(M\textbf{h}, M\textbf{g})$ is also contained in $U$, and hence $(M\textbf{h}, M\textbf{g}) = (B \textbf{f}, \Delta \textbf{f})$ for some $\textbf{f} \in H^2(Y)$. Then $$(\textbf{h}, \textbf{g}) = \Big(\textbf{B} \frac{\textbf{f}}{M} , \Delta \frac{\textbf{f}}{M}  \Big),$$ where $\textbf{f}/M \in N^+(Y)$.	
\end{proof}

\begin{thm}
\thlabel{conttheorem} Assume that $\h[\textbf{B}]$ consists of functions taking values in a finite dimensional Hilbert space. Then the set of functions in $\h[\textbf{B}]$ which extend continuously to $\cD$ is dense in the space.
\end{thm}

\begin{proof} Let $K = H^2(X) \oplus \deh$. Recall from \thref{modeltheorem} that the space $\h[\textbf{B}]$ is equipped with an isometric embedding $J$ where the tuple $J\textbf{f} = (\textbf{f},\textbf{g}) \in K$ is uniquely determined by the requirement for it be orthogonal to \[ U = \{ (\textbf{B}\textbf{h}, \Delta \textbf{h} : \textbf{h}\in H^2(Y)\} \subset K. \] We identify functions $\textbf{f} \in \h[\textbf{B}]$ with their coordinates $(f^1, \ldots, f^n)$ with respect to the fixed orthonormal basis of $X$. Now assume that $\textbf{f} \in \h[\textbf{B}]$ is orthogonal to any function in $\h[\textbf{B}]$ which extends continuously to $\cD$, i.e., that $\textbf{f}$ is orthogonal to $\h[\textbf{B}] \cap \A^n$. We shall show that $J\textbf{f} = (\textbf{B} \textbf{h}, \Delta \textbf{h})$ for some $\textbf{h} \in H^2(Y)$, which implies that $\textbf{f} = 0$. Consider $J(\A^n \cap \h[\textbf{B}])$ as a subspace of $\A^n \oplus \deh$. For each $\textbf{h} \in H^2(Y)$, let $$l_\textbf{h} = (\textbf{B} \textbf{h}, \Delta \textbf{h}) \in \mathcal{C}^n \oplus \deh$$ be a functional on $\A^n \oplus \deh$, acting as usual by integration on the boundary $\T$. We claim that $$J(\A^n \cap \h[\textbf{B}]) = \cap_{\textbf{h} \in H^2(Y)} \ker l_\textbf{h}.$$ Indeed, if $\textbf{f} \in \A^n \cap \h[\textbf{B}]$, then for any functional $l_\textbf{h}$ we have $$l_\textbf{h}(J\textbf{f}) = \ip{J\textbf{f}}{(\textbf{B} \textbf{h}, \Delta \textbf{h})} = 0,$$ because $J\textbf{f}$ is orthogonal to $U$. Conversely, if the tuple $(\textbf{f},\textbf{g}) \in \A^n \oplus \deh$ is contained in $\cap_{\textbf{h} \in H^2(Y)} \ker l_\textbf{h}$, then $(\textbf{f},\textbf{g}) \in K$ is orthogonal to $U$, and hence $\textbf{f} \in \A \cap \h[\textbf{B}]$ by \thref{modeltheorem}. Now, viewed as an element of $\mathcal{C}^n \oplus \deh$, the tuple $J\textbf{f}$ annihilates $J(\A^n \cap \h[\textbf{B}])$, and so by \thref{weakstarapprox} lies in the weak-star closure of linear manifold of functionals of the form $l_\textbf{h}$. Thus \thref{Sset} implies that $J\textbf{f} = (\textbf{B} \textbf{h}, \Delta \textbf{h})$ for some $\textbf{h} \in N^+(Y)$. The Smirnov maximum principle and the computation $$\int_\T \|\textbf{h}\|^2_Y d\m(\z) = \int_\T \|\textbf{B}\textbf{h}\|^2_X d\m  + \int_\T \|\Delta \textbf{h}\|^2_Y d\m < \infty$$ show that $\textbf{h} \in H^2(Y)$. Hence $J\textbf{f} \in (JH(\textbf{B})^\perp$, so that $\textbf{f} = 0$ and the proof is complete.	
\end{proof}

\thref{thm1} of the introduction now follows as a consequence of the next result, which is an easy extension of \thref{conttheorem}.

\begin{cor} \thlabel{conttheoremcor} Assume that $\h[\textbf{B}]$ consists of functions taking values in a finite dimensional Hilbert space. If $M$ is any $L$-invariant subspace of $\h[\textbf{B}]$, then the set of functions in $M$ which extend continuously to $\cD$ is dense in $M$.
\end{cor}

\begin{proof}
If $M$ contains the constant vectors, then \thref{kernel} applies, and hence $M$ is of the type $\h(\textbf{B}_0)$ for some contractive function $\textbf{B}_0$. Then the result follows immediately from \thref{conttheorem}. If constant vectors are not contained in $M$, then let $$M^+ = \{ \textbf{f} + c : \textbf{f} \in M, c \in X\}.$$ The subspace $M^+$ is closed, as it is a sum of a closed subspace and a finite dimensional space. Moreover, closed graph theorem implies that the skewed projection $P: M^+ \to M$ taking $\textbf{f}+c$ to $\textbf{f}$ is bounded. The theorem holds for $M^+$, so if $\textbf{f} \in M$, then there exists constants $c_n$ and functions $\textbf{f}_n \in M$ such that $\textbf{h}_n = \textbf{f}_n + c_n$ is continuous on $\cD$, and $\textbf{h}_n$ tends to $\textbf{f}$ in the norm of $\hil$. Consequently, the functions $\textbf{f}_n = \textbf{h}_n - c_n$ are continuous on $\cD$, and we have that $\textbf{f}_n = P\textbf{h}_n$ tends to $P\textbf{f} = \textbf{f}$ in the norm of $\hil$.
\end{proof}

\section{Applications of the density theorem} \label{applications1}

We temporarily leave the the main subject in order to present applications of \thref{conttheorem} and \thref{conttheoremcor}. All Hilbert spaces of analytic functions will be assumed to satisfy (A.1)-(A.3).

\subsection{$M_z$-invariant subspaces.} \thref{cor-beurling1} stated in the introduction is now an easy consequence of \thref{conttheorem}. We restate the theorem for the reader's convenience.

\begin{cor} Let $\h$ be a Hilbert space of analytic functions which satisfies (A.1)-(A.3)  and is invariant for the forward shift $M_z$. For a closed $M_z$-invariant subspace $\M$ of $\h$ with $\dim \M\ominus M_z\M=n<\infty$, let $\varphi_1,\ldots\varphi_n$ be an orthonormal basis in $\M\ominus M_z\M$, and denote by $\phi$ the corresponding row operator-valued function. Then \begin{equation}\M=\phi \h[\textbf{C}], \end{equation}
where $ \h[\textbf{C}]$ consists of $\mathbb{C}^n$-valued functions and the mapping $g \mapsto \phi g$ is an isometry from $\h[\textbf{C}]$ onto $\M$. Moreover, 
\begin{equation} \{\sum_{i=0}^n\varphi_iu_i:~u_i\in \A,~1\le i\le n\}\cap \h\end{equation}
is a dense subset of $\M$.
\end{cor}

\begin{proof}
For any $f \in  \M$ we have that $f(z) - \sum_{i=1}^n \ip{f}{\phi_i}_\h \phi_i(z) \in M_z\M$. Thus the operator $L^\phi: \M \to \M$ given by \[ L^\phi f(z) = \frac{f(z) - \sum_{i=1}^n \ip{f}{\phi_i}_\h \phi_i(z)}{z} \] is well-defined, and it is a contraction since it is a composition of a projection with the contractive operator $L$. A straightforward computation shows that for $\lambda \in \D$ the following equation holds \[ (1-\lambda L^\phi)^{-1} f(z) =  \frac{zf(z) - \lambda \sum_{i=1}^n \ip{(1-\lambda L^\phi)^{-1}f}{\phi_i}_\h\phi_i(z)}{z-\lambda}.\] Thus the analytic function in the numerator on the right-hand side above must have a zero at $z = \lambda$. It follows that $f(\lambda) = \sum_{i=1}^n \ip{(1-\lambda L^\phi)^{-1}f}{ \phi_i}_\h \phi_i(\lambda)$. Consider now the mapping $U$ taking $f \in \M$ to the vector $Uf(\lambda) = \Big( \ip{(1-\lambda L^\phi)^{-1}}{\phi_i}_\h \Big)_{i=1}^n$ and let $\M_0 = U\M$ with the norm on $\M_0$ which makes $U:\M \to \M_0$ a unitary mapping. Then $\M_0$ is a space of $\mathbb{C}^n$-valued analytic functions which satisfies (A.1')-(A.3') and to which \thref{conttheorem} applies. The claims in the statement follow immediately from this. 
\end{proof}

\subsection{Reverse Carleson measures.} A finite Borel measure on $\cD$ is a \textit{reverse Carleson measure} for $\h$ if there exists a constant $C > 0$ such that the estimate \begin{equation}\|f\|^2_\h \leq C \int_{\cD} |f(z)|^2 d\mu(z) \label{RCineq}\end{equation} holds for $f$ which belong to some dense subset of $\h$ and for which the integral on the right-hand side makes sense, e.g. by the existence of radial boundary values of $f$ on the support of the singular part of $\mu$ on $\T$. For the class of spaces considered in this paper it is natural to require, due to \thref{conttheorem}, that \eqref{RCineq} holds for all functions in $\h$ which admit continuous extensions to $\cD$. 

Our main result in this context characterizes the existence of a reverse Carleson measures for spaces which are invariant for $M_z$. If such a measure exists, then we can moreover identify one which is in a sense minimal. 

\begin{thm} \thlabel{revcarl2}
Let $\h$ be invariant for $M_z$. Then the following are equivalent.
\begin{enumerate}[(i)]
	\item $\h$ admits a reverse Carleson measure.
	\item  $$\sup_{0 < r < 1} \int_\T \Bigg\| \frac{\sqrt{1-|r\lambda|^2}}{1-r\conj{\lambda}z}\Bigg\|^2_{\h} d\m(\lambda) < \infty.$$
	\item If $k$ is the reproducing kernel of $\h$, then$$\sup_{0 < r < 1} \int_\T \frac{1}{(1-|r\lambda|^2)k(r\lambda,r\lambda)} d\m(\lambda)  < \infty.$$
\end{enumerate}
	
If the above conditions are satisfied, then $$h_1(\lambda) := \lim_{r \rightarrow 1} \Bigg\| \frac{\sqrt{1-|r\lambda|^2}}{1-r\conj{\lambda}z}\Bigg\|^2_{\h} $$ and $$h_2(\lambda) := \lim_{r \rightarrow 1} \frac{1}{(1-|r\lambda|^2)k(r\lambda, r\lambda)}$$ define reverse Carleson measures  for $\hil$. Moreover, if $\nu$ is any reverse Carleson measure for $\hil$ and $v$ is the density of the absolutely continuous part of the restriction of $\nu$ to $\T$, then $h_1d\m$ and $h_2d\m$ have the following minimality property: there exist constants $C_i > 0, i =1,2$ such that $$h_i(\lambda) \leq C_i v(\lambda)$$ for almost every $\lambda \in \T.$
\end{thm}

\begin{proof}(i) $\Rightarrow$ (ii): Let $\nu$ be a reverse Carleson measure for $\h$. First, we show that we can assume that $\nu$ is supported on $\T$. For this, we will use the inequality $$\|z^nf\|^2_\h \leq C \int_\D |z^nf(z)|^2 d\nu(z) + C \int_{\T} |f(z)|^2 d\nu(z).$$ Since $Lzf = f$ and $L$ is a contraction, it follows that $\|zf\|_\h \geq \|f\|_\h$, and thus letting $n$ tend to infinity in the above inequality we obtain $$ \|f\|^2_\h \leq C \int_{\T} |f(z)|^2 d(\nu|\T)(z).$$ Thus we might replace $\nu$ by $\nu|\T$, as claimed. Next, we note that $H^\infty \subset \h$. Indeed, $\h$ contains $1$ and is $M_z$-invariant, thus contains the polynomials. If $p_n$ is a uniformly bounded sequence of polynomials converging pointwise to $f \in H^\infty$, then the existence of a reverse Carleson measure ensures that the norms $\|p_n\|_\h$ are uniformly bounded, and thus a subsequence of $\{p_n\}$ converges weakly to $f \in \h$. Thus, the function $z \mapsto \frac{1}{1-\conj{\lambda}z}$ is contained in $\h$ for each $\lambda \in \D$. Define $$H(\lambda) =: \int_\T \frac{1-|\lambda|^2}{|1-\conj{\lambda}z|^2} d\nu(z), \quad \lambda \in \D,$$ which is positive and harmonic in $\D$, and since $\nu$ is a reverse Carleson measure for $\hil$, there exists a constant $C > 0$ such that \begin{equation} \Bigg\|\frac{\sqrt{1-|\lambda|^2}}{1-\conj{\lambda}z}\Bigg\|_\h^2 \leq CH(\lambda). \label{RCineq2} \end{equation} The implication now follows from the mean value property of harmonic functions. 

(ii) $\Rightarrow$ (iii): There exists an orthogonal decomposition $$\frac{1}{1-\conj{\lambda}z} = \frac{1}{1-|\lambda|^2}\frac{k(\lambda, z)}{k(\lambda, \lambda)} + g(z),$$ where $g$ is some function which vanishes at $\lambda$. Thus $$\Big\|\frac{1}{1-\conj{\lambda}z}\Big\|_\h^2 = \frac{1}{(1-|\lambda|^2)^2k(\lambda,\lambda)} + \|g\|^2$$ and consequently \begin{equation}\frac{1}{(1-|\lambda|^2)k(\lambda,\lambda)} \leq \Bigg\|\frac{\sqrt{1-|a|^2}}{1-\conj{\lambda}z}\Bigg\|_\h^2. \label{RCineq4} \end{equation} Thus (iii) follows from (ii). 

(iii) $\Rightarrow$ (i): If $\h = \h[\textbf{B}]$ with $\textbf{B} = (b_i)_{i=1}^\infty$, then let $w$ be the outer function with boundary values satisfying $|w(\z)|^2 = 1-\sum_{i\in I} |b_i(\z)|^2$. The existence of such a function is ensured by (iii). The space $M = wH^2 = \{f = wg : g \in H^2 \}$ normed by $\|f\|_{M} = \|g\|_2$ is a Hilbert space of analytic functions with a reproducing kernel given by  $$K_{M}(\lambda, z) = \frac{\conj{w(\lambda)}w(z)}{1-\conj{\lambda}z}.$$ It is not hard to verify that $K = k_\h - K_M$ is a positive-definite kernel. Then $k_\h = K + K_M$, and hence $M$ is contained contractively in $\h$. Thus for any function $f \in M$ we have that $\|f\|^2_{\hil} \leq \|f\|^2_M = \int_\T \frac{|f|^2}{|w|^2} d\m.$, and $|w|^{-2}d\m$ will be a reverse Carleson measure if $M$ is dense in $\hil$. But $1/w \in H^2$ by (iii), and thus $\h \cap \A \subset M$, so $M$ is indeed dense in $\h$. 

The limits defining $h_1$ and $h_2$ exist as a consequence of general theory of boundary behaviour of subharmonic functions. The inequality $h_1(\lambda) \leq C_1v(\lambda)$ is seen immediately from \eqref{RCineq2} and $h_1(\lambda) \leq C_2v(\lambda)$ is then seen from \eqref{RCineq4}.
\end{proof}

We remark that if $\h$ is not $M_z$-invariant, then the space might admit a reverse Carleson measure even though (iii) is violated. An example is any $L$-invariant proper subspace of $H^2$.

An application of part (iii) of \thref{revcarl2} to $\h = \h(b)$, with $b$ non-extreme, lets us deduce a result essentially contained in \cite{revcarlesonross}, namely that $\h(b)$ admits a reverse Carleson measure if and only if $(1-|b|^2)^{-1} \in L^1(\T)$, and the measure $\mu = (1-|b|^2)^{-1}d\m$ is then a minimal reverse Carleson measure in the sense made precise by the theorem. A second application is to Dirichlet-type spaces. Recall from Section \ref{examplessubsec} that for $\mu$ a positive finite Borel measure supported on $\cD$, the Hilbert space $\mathcal{D}(\mu)$ is defined as the completion of the analytic polynomials under the norm $$\|f\|^2_{\mathcal{D}(\mu)} = \|f\|^2_2 + \int_{\T} \int_{\cD} \frac{|f(z)-f(\lambda)|^2}{|z-\lambda|^2}  d\mu(z) d\m(\lambda).$$

\begin{cor}
The space $D(\mu)$ admits a reverse Carleson measure if and only if $$\int_{\cD} \frac{d\mu(z)}{1-|z|^2} < \infty.$$ The minimal reverse Carleson measure (in the sense of \thref{revcarl2}) is given by $d\nu = hd\m$ where $$h(\lambda) = 1 + \int_{\cD} \frac{d\mu(z)}{|1 - \conj{\lambda}z|^2}, \lambda \in \T.$$
\end{cor}

\begin{proof} The space $\mathcal{D}(\mu)$ satisfies the assumptions of \thref{revcarl2}. A computation shows that if $k_\lambda(z) = \frac{1}{1-\conj{\lambda}z}$, then $$(1-|\lambda|^2)\|k_\lambda\|^2_{\mathcal{D}(\mu)} = 1 + \int_{\cD} \frac{|\lambda|^2}{|1-\conj{\lambda}z|^2}d\mu(z).$$ The claim now follows easily from (ii) of \thref{revcarl2} and Fubini's theorem.
\end{proof}

It is interesting to note that the condition on $\mu$ above holds even in cases when $\mathcal{D}(\mu)$ is strictly contained in $H^2$ (see \cite{alemanhabil}).

Our last result in the context of reverse Carleson measures is a non-existence result which answers in particular a question posed in \cite{revcarlesonross}.

\begin{thm}\thlabel{rctheorem}	
	Assume that the identity $$\|Lf\|^2_{\h} = \|f\|^2_{\h} - |f(0)|^2$$ holds in $\h$. If $\h$ admits a reverse Carleson measure, then $\h$ is isometrically contained in the Hardy space $H^2$. 
\end{thm}

\begin{proof}
	We will use \thref{modeltheorem} and \thref{modelformulaconnection}. The limit in part (ii) of \thref{modelformulaconnection} vanishes, because the norm identity implies that $\|zL_\lambda f\| = \|L_\lambda f\|_{\h}$. It will thus suffice to show that the limit in part (iii) of said theorem also vanishes, namely that \begin{gather} \lim_{r \to 1} \int_{\T} (1-r^2)\|L_{r\lambda}f\|^2_{\h} d\m(\lambda) = 0. \label{ineq2}\end{gather} If $f \in \h$ is continuous in $\cD$ then so is $L_{r\lambda}f$, and if $\mu$ is a reverse Carleson measure for $\h$, then we have the estimate
	\begin{gather}
	\lim_{r \to 1}\int_{\T} (1-r^2)\|L_{r\lambda}f\|^2_{\h} d\m(\lambda) \nonumber \\ \leq C \limsup_{r\to 1}\int_{\cD} \int_\T \frac{1-r^2}{|\z - r\lambda|^2}|f(\z) - f(r\lambda)|^2 d\m(\lambda)  d\mu(\z) \label{ineq3}
	\end{gather} Since $f$ is continuous and bounded in $\cD$, we readily deduce from standard properties of the Poisson kernel and bounded analytic functions that the functions $$G_r(\z) = \int_\T \frac{1-r^2}{|\z - r\lambda|^2}|f(\z) - f(r\lambda)|^2 d\m(\lambda)$$ are uniformly bounded in $\cD$ and that $\lim_{r \to 1} G_r(\z) = 0$ for all $\z \in \cD$. Dominated convergence theorem now implies that the limit in \eqref{ineq3} is 0, and so \eqref{ineq2} holds.
\end{proof}

From part (ii) of \thref{modeltheorem} we easily deduce that the condition on $L$ of \thref{rctheorem} is equivalent to $\Theta H^2(Y_1) = \{0\}$ in \eqref{decomp}. In the case that $\textbf{B} = b$ is a scalar-valued function, this occurs if and only if $b$ is an extreme point of the unit ball of $H^\infty$. Thus an extreme point $b$ which is not an inner function cannot generate a space $\h(b)$ which admits a reverse Carleson measure. We remark also that \thref{rctheorem} holds, with the same proof, even when $\h$ consists of functions taking values in a finite dimensional Hilbert space $X$, where the norm identity then instead reads $\|L\textbf{f}\|^2 = \|\textbf{f}\|^2 - \|\textbf{f}(0)\|^2_X$, and definition of reverse Carleson measure is extended naturally to the vector-valued setting.

\subsection{Formula for the norm in a nearly invariant subspace.} A space $\M$ is \textit{nearly invariant} if whenever $\lambda \in \D$ is not a common zero of the functions in $\M$ and $f(\lambda) = 0$ for some $f \in \M$, then $\frac{f(z)}{z-\lambda} \in \M$. If $\h$ is $M_z$-invariant, then an example of a nearly invariant subspace is any $M_z$-invariant subspace for which $\dim \M \ominus M_z\M = 1$. The concept of a nearly invariant subspace has appeared in \cite{hittanulus}, and have since been used as a tool in solutions to numerous problems in operator theory. 

We will now prove a formula for the norm of functions contained in a nearly invariant subspace $\M$ which is similar to \thref{formula} but which is better suited for exploring the structure of $\M$. 

\begin{prop} \thlabel{nearinvnorm} Let $\M \subseteq \h$ be a nearly invariant subspace and $k$ be the common order of the zero at $0$ of the functions in $\M$. Let $\phi \in \M$ be the function satisfying $\ip{f}{\phi}_\h = \frac{f^{(k)}(0)}{\phi^{(k)}(0)}$ for all $f \in \M$, $L^\phi: \M \to \M$ be the contractive operator given by $$L^\phi f(z) = \frac{f(z)-\ip{f}{\phi}_\h\phi(z)}{z},$$ and $L^\phi_\lambda = L^\phi(1-\lambda L^\phi)^{-1}, \lambda \in \D$. If the sequence $\{L^n\}_{n=1}^\infty$ converges to zero in the strong operator topology, then \begin{equation}
\|f\|^2_{\h} = \|f/\phi\|^2_2 + \lim_{r \to 1} \int_{\T} \|zL^\phi_{r\lambda} f\|^2_{\h}-\|L^\phi_{r\lambda} f\|^2_{\h} d\m(\lambda). \label{eq221}
\end{equation}
\end{prop}

\begin{proof}
The fact that $L^\phi$ maps $\M$ into itself follows easily by nearly invariance of $\M$. The formula	
\begin{equation}
\|f\|^2_{\h} = \|f/\phi\|^2_2 + \lim_{r \to 1} \int_{\T} \|zL^\phi_{r\lambda} f\|^2_{\h}-r^2\|L^\phi_{r\lambda} f\|^2_{\h} d\m(\lambda) \label{eq211}
\end{equation} holds in much more general context and without the assumption on convergence of $\{L^n\}_{n=1}^\infty$ to zero. For its derivation we refer to Lemma 2.2 of \cite{alemanrichtersimplyinvariant}, and the discussion succeeding it. The equation \eqref{eq221} will follow from \eqref{eq211} if we can show that the additional assumption on $L$ implies that $$\lim_{r \to 1} \int_\T (1-r^2) \|L^\phi_{r\lambda} f\|^2_{\h} d\m(\lambda) = 0.$$ To this end, observe that for $\lambda \in \D$ we have $$L^\phi_{\lambda}f (z) = \frac{f(z) - f(\lambda)}{z-\lambda}  - \frac{f(\lambda)}{\phi(\lambda)} \frac{\phi(z)- \phi(\lambda)}{z-\lambda}  = L_\lambda f(z) - \frac{f(\lambda)}{\phi(\lambda)}L_\lambda \phi(z),$$ and so if $f/\phi \in H^\infty$, then the argument in the proof of part (iii) of \thref{modelformulaconnection} shows that \begin{gather*}
\lim_{r \to 1^-} \int_\T (1-r^2) \|L^\phi_{r\lambda} f\|^2_{\h} d\m(\lambda) \\ \lesssim \lim_{r \to 1^-} \int_\T (1-r^2) \|L_{r\lambda} f\|^2_{\h} d\m(\lambda) + \|f/\phi\|^2_\infty \int_\T (1-r^2) \|L_{r\lambda} \phi \|^2_{\h} d\m(\lambda) = 0. \end{gather*}  
Next, the Hilbert space $\tilde{\M} = \M/\phi = \{f/\phi : f \in \M\}$ with the norm $\|f/\phi\|_{\tilde{\M}} = \|f\|_{\h}.$ It is easy to see that $\tilde{\M}$ satisfies (A.1)-(A.3), and thus by \thref{conttheorem} the functions with continuous extensions to $\cD$ form a dense subset of $\tilde{\M}$. Since multiplication by $\phi$ is a unitary map from $\tilde{\M}$ to $\M$, we see that functions $f \in \M$ such that $f/\phi$ has continuous extension to $\cD$ form a dense subset of $\M$.	Finally, consider the mapping \[\textbf{Q}f = \limsup_{r \to 1^{-}} \int_\T (1-r^2) \|L^\phi_{r\lambda} f\|^2_{\h} d\m(\lambda), \quad f \in \M.\] Then $\textbf{Q}(f+g) \leq \textbf{Q}f + \textbf{Q}g$, and we have shown above that $\textbf{Q} \equiv 0$ on a dense subset of $\M$. A peek at \eqref{eq211} reveals that $\textbf{Q}f \leq \|f\|^2_{\h}$, and so $\textbf{Q}$ is continuous on $\M$. Thus $\textbf{Q} \equiv 0$ since it vanishes on a dense subset.
\end{proof}


We will see an application of \thref{nearinvnorm} in the sequel. For now, we show how it can be used to deduce a theorem of Hitt on the structure of nearly invariant subspaces of $H^2$ (see \cite{hittanulus}). 

\begin{cor}
If a closed subspace $\M \subset H^2$ is nearly invariant, then it is of the form $\M = \phi K_\theta$, where $K_\theta = H^2 \ominus \theta H^2$ is an $L$-invariant subspace of $H^2$, and we have the norm equality $\|f/\phi\|_2 = \|f\|_2, f \in \M$.  
\end{cor}

\begin{proof}
If $\M$ is nearly invariant then the formula \eqref{eq221} gives $\|f/\phi\|_2 = \|f\|$, since $M_z$ is an isometry on $H^2$. Thus $\M/\phi := \{f/\phi : f \in \M\}$ is closed in $H^2$, and it is easy to see that it is $L$-invariant. Thus $\M/\phi = K_\theta$ by Beurling's famous characterization.
\end{proof}

\subsection{Orthocomplements of shift invariant subspaces of the Bergman space.}
The Bergman space $L^2_a(\D)$ consists of functions $g(z) = \sum_{k=0} g_kz^k$ analytic in $\D$ which satisfy $$\|g\|^2_{L^2_a(\D)} = \int_\D |g(z)|^2 dA(z) = \sum_{k=0}^\infty (k+1)^{-1}|g_k|^2 < \infty,$$ where $dA$ denotes the normalized area measure on $\D$. The Bergman space is invariant for $M_z$ and the lattice of $M_z$-invariant subspaces of $L^2_a(\D)$ is well-known to be very complicated (see for example Chapter 8 and 9 of \cite{durenbergmanspaces} and Chapters 6, 7 and 8 of \cite{hedenmalmbergmanspaces}). Before stating and proving our next result, we will motivate it by showing that the orthogonal complements of $M_z$-invariant subspaces of $L^2_a(\D)$ can consist entirely of ill-behaved functions. The following result is essentially due to A. Borichev \cite{borichevpriv}.

\begin{prop} \thlabel{badorthocomplements}
	There exists a subspace $\M \subsetneq L^2_a(\D)$ which is invariant for $M_z$, with the property that for any non-zero function $g \in \M^\perp$ and any $\delta > 0$ we have $$\int_{\D} |g|^{2+\delta} dA = \infty.$$
\end{prop}

\begin{proof} The argument is based on the existence of a function $f$ such that $f, 1/f \in L^2_a(\D)$, yet if $\M$ is the smallest $M_z$-invariant subspace containing $f$, then $\M \neq L^2_a(\D)$. Such a function exists by \cite{borichevhedenmalmcyclicity}. Take $g \in \M^\perp$ and assume that $\int_\D |g|^{2+\delta} dA < \infty$ for some $\delta > 0$. Since $\M$ is $M_z$-invariant, for any polynomial $p$ we have that \begin{equation} \int_\D pf\conj{g} dA = 0.\label{eq37} \end{equation} By the assumption on $g$ and H\"older's inequality we have that $f\conj{g} \in L^r(\D)$ for $r > 1$ sufficiently close to 1. Let $s = \frac{r}{r-1}$ be the H\"older conjugate index of $r$. Since $1/f \in L^2_a(\D)$, the analytic function $f^{-\epsilon}$ is in the Bergman space $L^s_a(\D)$ for sufficiently small $\epsilon > 0$, and hence can be approximated in the norm of $L^s_a(\D)$ by a sequence $\{p_n\}_{n=1}^\infty$ of polynomials. If $p$ is any polynomial, then \begin{gather*}\lim_{n\to \infty} \int_\D |(p_n-f^{-\epsilon})pfg| dA \\ \lesssim \lim_{n\to \infty} \|p_n-f^{-\epsilon}\|_{L^s(\D)}\|pfg\|_{L^r(\D)} = 0.\end{gather*} By choosing $\epsilon = 1/\M$ for sufficiently large positive integer $\M$, we see from the above that for any polynomial $p$ we have $$\int_\D pf^{1-1/\M}\conj{g}dA = 0,$$ which is precisely \eqref{eq37} with $\conj{g}$ replaced by $f^{-1/\M}\conj{g}$, and of course we still have $ff^{-1/\M}\conj{g} \in L^r(\D)$ for the same choice of $r > 1$. Repeating the argument gives $$\int_\D pf^{1-2/\M}\conj{g} dA = 0,$$ and after $\M$ repetitions of the argument we arrive at $$\int_\D p\conj{g} dA(z) = 0$$ for any polynomial $p$. Then $g = 0$ by density of polynomials in $L^2_a(\D)$. 
\end{proof}
Despite the rather dramatic situation described by \thref{badorthocomplements}, an application of \thref{conttheoremcor} yields the following result. 

\begin{cor} \thlabel{contderivdense}
	Let $\M$ be a subspace of $L^2_a(\D)$ which is invariant for $M_z$. Then the functions in the orthocomplement $\M^\perp$ which are derivatives of functions in the disk algebra $\A$ are dense in $\M^\perp$. 
\end{cor}

\begin{proof} The operator $U$ given by $$Uf(z) = \frac{1}{z}\int_0^z f(w) \,dw$$ is a unitary map between $L^2_a(\D)$ and the classical Dirichlet space $\mathcal{D}$, and \thref{conttheoremcor} applies to the latter space. A computation involving the Taylor series coefficients shows that $UM_z^*U^* = L$, and so if $\M \subset L^2_a(\D)$ is $M_z$-invariant, then $U\M^\perp \subset \mathcal{D}$ is $L$-invariant. Thus from \thref{conttheoremcor} we infer that the set of functions $f \in \M^\perp$ for which $Uf(z)$ is in the disk algebra $\A$ is dense in $\M^\perp$, and for any function $f$ in this set we have $f(z) = (zUf(z))'$, so $f$ is the derivative of a function in $\A$. 
\end{proof}

\section{Finite rank $\h[\textbf{B}]$-spaces}

\subsection{Finite rank spaces.} 
The applicability of \thref{modeltheorem} depends highly on the ability to identify the spaces $W$ and $\Theta H^2(Y_1)$ in \eqref{decomp}. In general, $\Delta(\z)$ defined by \eqref{delta} is taking values in the algebra of operators on an infinite dimensional Hilbert space. The situation is much more tractable in the case of finite rank $\h[\textbf{B}]$-spaces, for then $\Delta(\z)$ acts on a finite dimensional space. In this last section we restrict ourselves to the study of the finite rank case. Thus, we study spaces of the form $\h = \h[\textbf{B}]$ with $\textbf{B} = (b_1, \ldots, b_n)$, where $n$ is the rank of $\h[\textbf{B}]$ was defined in the introduction. It follows that $b_1, \ldots, b_n$ are linearly independent. For convenience, we will also be assuming that $\textbf{B} \neq 0,1$, which correspond to the cases $\h[\textbf{B}] = H^2$ and $\h(\textbf{B} = \{0\}$, respectively. 

%

\subsection{$M_z$-invariance and consequences.} It turns out that the decomposition \eqref{decomp} is particularly simple in the case when $\h[\textbf{B}]$ is invariant for $M_z$.

\begin{prop} \thlabel{noWlemma}
	If the finite rank $\h[\textbf{B}]$-space is invariant for $M_z$, then $W = \{0\}$ in the decomposition \eqref{decomp}, and thus \thref{modeltheoremanalytic} applies, with $\textbf{A}$ a square matrix and $\det \textbf{A} \neq 0$. 
\end{prop}

\begin{proof}
In the notation of \thref{modeltheorem} we have $Y \simeq \mathbb{C}^n$, thus \eqref{decomp} becomes \begin{equation} 
\conj{\Delta H^2(\mathbb{C}^n)} = W \oplus \Theta H^2(\mathbb{C}^m), \label{eq53}
\end{equation} where $\quad m \leq n$. We will see that  $W = \{0\}$ by showing something stronger, namely that $m = n$ in \eqref{eq53}. To this end, let $f \in \h[\textbf{B}]$ and consider $Jf = (f,\textbf{g})$ and $JM_zf = (M_zf, \textbf{g}_0)$, $J$ being the embedding given by \thref{modeltheorem}. Let $\textbf{g} = \textbf{w} + \Theta \textbf{h}$ and $\textbf{g}_0 = \textbf{w}_0 + \Theta \textbf{h}_0$ be the decompositions with respect to \eqref{eq53}. Since $LM_zf = f$, we see from part (ii) of \thref{modeltheorem} that $\textbf{w}_0(\z) = \z\textbf{w}(\z)$ and $\textbf{h}_0(\z) = \z\textbf{h}(\z) + \textbf{c}_f$, where $\textbf{c}_f \in \mathbb{C}^m$. By part (i) of \thref{modeltheorem} we have that \begin{gather} \textbf{B}(\zeta)^*\z\textbf{f}(\z) + \Delta(\z) \textbf{g}_0(\z) \nonumber \\ = \z\Big(\textbf{B}(\z)^*\textbf{f}(\z) + \Delta(\z)\textbf{g}(\z)\Big) + \textbf{A}(\z)^*\textbf{c}_f  \in \conj{H^2_0(\mathbb{C}^n)} \label{eq41}, \end{gather} where, as before, $\textbf{A} = \Theta^*\Delta$. We apply this to the reproducing kernel $k_\lambda(z)$ of $\h[\textbf{B}]$. Recall from Section \ref{prelimsec} that \begin{equation*}
Jk_\lambda = \Bigg( \frac{1 - \textbf{B}(z)\textbf{B}(\lambda)^*}{1-\conj{\lambda}z}, \frac{-\Delta(\z)\textbf{B}(\lambda)^*}{1-\conj{\lambda}z}\Bigg) = (k_\lambda, \textbf{g}_\lambda).
\end{equation*} A brief computations shows that $$\textbf{B}^*(\z)k_\lambda(\z) + \Delta(\z)\textbf{g}_\lambda(\z) = \conj{\z}\frac{\textbf{B}(\z)^* - \textbf{B}(\lambda)^*}{\conj{\z}-\lambda}.$$ Using \eqref{eq41} we deduce that for each $\lambda \in \D$ there exists $\textbf{c}_\lambda \in \mathbb{C}^m$ such that $$\frac{\textbf{B}(\z)^* - \textbf{B}(\lambda)^*}{\conj{\z}-\lambda} + \textbf{A}(\z)^*\textbf{c}_\lambda \in \conj{H^2_0(\mathbb{C}^n)}.$$ Since $\textbf{B}(0) = 0$ we see that the constant term of the above function equals $$0 = -\textbf{B}(\lambda)^*/\lambda + \textbf{A}(0)^*\textbf{c}_\lambda.$$ By linear independence of the coordinates $\{b_i\}_{i=1}^n$ we conclude that $\textbf{A}(0)^*$ maps an $m$-dimensional vector space onto an $n$-dimensional vector space. Thus it follows that $m = n$, and hence $W = \{0\}$ in \eqref{eq53}. Since $\textbf{A}$ is outer, invertibility of $\textbf{A}(0)$ implies that $\det \textbf{A} \neq 0$, and thus the function $\det \textbf{A}$ is non-zero and outer. 
\end{proof}

The following result characterizes $M_z$-invariance in terms of modulus of $\textbf{B}$, and is a generalization of a well-known theorem for $\h(b)$-spaces.

\begin{thm} \thlabel{shiftinvcriterion}
A finite rank $\h[\textbf{B}]$-space is invariant under the forward shift operator $M_z$ if and only if $$\int_\T \log(1-\|\textbf{B}\|^2_2) d\m = \int_\T \log(1-\sum_{i=1}^n |b_i|^2) d\m > -\infty.$$
\end{thm}

\begin{proof} Assume that $\h[\textbf{B}]$ is $M_z$-invariant. Then $\det \textbf{A}(z)$ is non-zero by \thref{noWlemma}. In terms of boundary values on $\T$ we have $\textbf{A} = \Theta^* \Delta = \Theta^*(I_n - \textbf{B}^*\textbf{B})^{1/2}$, where $I_n$ is the $n$-by-$n$ identity matrix, and so $\det \textbf{A} = \conj{\det \Theta}(1-\sum_{i=1}^n |b_i|^2)^{1/2}$ on $\T$. Since $\Theta$ is an isometry we have that $|\det \Theta| = 1$, and thus $$\int_\T \log(1-\sum_{i=1}^n |b_i|^2) d\m = \int_\T \log(|\det \textbf{A}|) d\m> -\infty,$$ last inequality being a well-known fact for bounded analytic functions.
	
Conversely, assume that $\int_\T \log(1-\sum_{i=1}^n |b_i|^2) d\m > -\infty$. Thus $1-\sum_{i=1}^n |b_i|^2 = \det \Delta  > 0$ almost everywhere on $\T$. Consider again the decomposition in \eqref{eq53}. We claim that $W = \{0\}$. Assume, seeking a contradiction, that $W \neq \{0\}$. $W$ is invariant under multiplication by scalar-valued bounded measurable functions, and so $W$ contains a function $\textbf{g}$ which is non-zero but vanishes on a set of positive measure. Fix $\textbf{h}_n \in H^2(\mathbb{C}^n)$ such that $\Delta \textbf{h}_n \to \textbf{g}$. Let $\text{adj} (\Delta) := (\det \Delta)\Delta^{-1}$ be the adjugate matrix. Then $\text{adj} (\Delta)$ has bounded entries and thus \begin{equation}\det \Delta \textbf{h}_n = \text{adj} (\Delta) \Delta \textbf{h}_n \to \text{adj} (\Delta) \textbf{g},\label{eq54}\end{equation} in $L^2(\mathbb{C}^n)$. By our assumption, there exists an analytic outer function $d$ with $|d| = |\det \Delta|= (1-\sum_{i=1}^n |b_i|^2)^{1/2}$ almost everywhere on $\T$. Then $\det \Delta = \psi d$ for some measurable function $\psi$ of modulus $1$ almost everywhere on $\T$, and $\det \Delta \textbf{h}_n = \psi d\textbf{h}_n \in \psi H^2(\mathbb{C}^n),$ where $\psi H^2(\mathbb{C}^n)$ is norm-closed and contains no non-zero function which vanishes on a set of positive measure on $\T$. But by \eqref{eq54} we have $\text{adj} (\Delta) \textbf{g} \in \psi H^2(\mathbb{C}^n)$, which is a contradiction. Thus $W = \{0\}$ in \eqref{eq53}, so that $\conj{\Delta H^2(\mathbb{C}^n)} = \Theta H^2(\mathbb{C}^m)$ and \thref{modeltheoremanalytic} applies. We have $1-\sum_{i=1}^n |b_i(\z)|^2 > 0$ almost everywhere on $\T$, and therefore $\Delta (\z)$ is invertible almost everywhere on $\T$. This implies that $m = n$. Since $\textbf{A}(\z) = \Theta(\z)^* \Delta(\z)$, we see that $\textbf{A}(z)$ is an $n$-by-$n$ matrix-valued outer function. In particular, $\textbf{A}(z)$ is invertible at every $z \in \D$. Let $J$ be the embedding of \thref{modeltheoremanalytic} and $Jf = (f, \textbf{f}_1)$, where $f \in \h[\textbf{B}]$ is arbitrary. We claim that we can find a vector $\textbf{c}_f \in \mathbb{C}^n$ such that \begin{equation}\textbf{B}(\zeta)^*\z f(\z) + \textbf{A}(\z)^* (\z\textbf{f}_1(\z) + \textbf{c}_f) \in \conj{H^2_0(\mathbb{C}^n)}.\label{eq554} \end{equation} Indeed, by (i) of \thref{modeltheoremanalytic} we have that $\textbf{B}(\z)^*f + \textbf{A}(\z)^*\textbf{f}_1(\z) \in \conj{H^2_0(\mathbb{C}^n)}.$ Let $\textbf{v}$ be the zeroth-order coefficient of the coanalytic function $\z \textbf{B}(\z)^*f + \z \textbf{A}(\z)^*\textbf{f}_1(\z)$. It then suffices to take $\textbf{c}_f = (\textbf{A}(0)^*)^{-1}\textbf{v}$ for \eqref{eq554} to hold. Thus by (i) of \thref{modeltheoremanalytic} we have that $zf(z) \in \h[\textbf{B}]$, which completes the proof.
\end{proof}

From now on we will be working exclusively with $M_z$-invariant $\h[\textbf{B}]$-spaces of finite rank, and $J$ will always denote the embedding of $\h[\textbf{B}]$-space that appears in \thref{modeltheoremanalytic}.

\begin{lemma} \thlabel{shiftspecrad}
	If a finite rank $\h[\textbf{B}]$-space is $M_z$-invariant, then the spectral radius of the operator $M_z$ equals 1.
\end{lemma}

\begin{proof} We have seen above that if $Jf = (f, \textbf{f}_1)$, then $JM_zf = (zf(z), z\textbf{f}_1(z) + \textbf{c}_f),$ where $\textbf{c}_f \in \mathbb{C}^n$ is some vector depending on $f$. This shows that $M_z$ is unitarily equivalent to a finite rank perturbation of the isometric shift operator $(g, \textbf{g}_1) \mapsto (zg, z\textbf{g}_1)$ acting on $H^2 \oplus H^2(\mathbb{C}^n)$. It follows that the essential spectra of the two operators coincide, and so are contained in $\cD$. Thus for $|\lambda| > 1$ the operator $M_z - \lambda$ is Fredholm of index $0$. Since $M_z - \lambda$ is injective, index $0$ implies that that $M_z - \lambda$ is invertible.
\end{proof}

A consequence of \thref{shiftspecrad} is that the operators $(I_{\h[\textbf{B}]} - \conj{\lambda} M_z)^{-1}$ exist for $|\lambda| < 1$. Their action is given by $$(I_{\h[\textbf{B}]} - \conj{\lambda} M_z)^{-1}f(z) = \frac{f(z)}{1-\conj{\lambda}z}.$$ Let $$J\frac{f(z)}{1-\conj{\lambda}z} = \Big(\frac{f(z)}{1-\conj{\lambda}z}, \textbf{g}_{\lambda}(z)\Big)$$ for some $\textbf{g}_\lambda \in H^2(\mathbb{C}^n)$. We compute \begin{gather*}Jf = J(I_{\h[\textbf{B}]}-\conj{\lambda}M_z)(I_{\h[\textbf{B}]}-\conj{\lambda}M_z)^{-1}f  = \Big( f, (1-\conj{\lambda}z)\textbf{g}_\lambda(z) + \textbf{c}_f(\lambda) \Big), \end{gather*} where $\textbf{c}_f(\lambda) \in \mathbb{C}^n$ is some vector depending on $f$ and $\lambda$. By re-arranging we obtain $\textbf{g}_\lambda(z) = \frac{\textbf{f}_1(z) - \textbf{c}_f(\lambda)}{1-\conj{\lambda}z},$
and so \begin{equation}
	J\frac{f(z)}{1-\conj{\lambda}z} = \Big(\frac{f(z)}{1-\conj{\lambda}z}, \frac{\textbf{f}_1(z) - \textbf{c}_f (\lambda)}{1-\conj{\lambda}z}\Big). \label{eq43}
\end{equation} We will now show that $\textbf{c}_f(\lambda)$ is actually a coanalytic function of $\lambda$ which admits non-tangential boundary values almost everywhere on $\T$. To this end, let $\textbf{u}_f(z) \in \conj{H^2_0(\mathbb{C}^n)}$ be the coanalytic function with boundary values $$\textbf{u}_f(\z) = \textbf{B}(\z)^*f(\z) + \textbf{A}(\z)^*\textbf{f}_1(\z).$$ By (i) of \thref{modeltheoremanalytic} and \eqref{eq43} we have 
\begin{equation}
	\frac{\textbf{B}(\z)^*f(\z) + \textbf{A}(\z)^*\textbf{f}_1(\z)}{1 - \conj{\lambda}\z} - \frac{\textbf{A}(\z)^*\textbf{c}_f(\lambda)}{1-\conj{\lambda}\z} = \frac{\textbf{u}_f(\z)}{1 - \conj{\lambda}\z} - \frac{\textbf{A}(\z)^*\textbf{c}_f(\lambda)}{1-\conj{\lambda}\z}\in \conj{H^2_0(\mathbb{C}^n)}, \label{eq44}
\end{equation} and projecting this equation onto $H^2(\mathbb{C}^n)$ we obtain \begin{equation} \textbf{u}_f(\lambda) = A(\lambda)^*\textbf{c}_f(\lambda).
\label{eq45}
\end{equation} Since $\textbf{A}(\lambda)^*$ is coanalytic and invertible for every $\lambda \in \D$, we get that $$\textbf{c}_f(\lambda) = (\textbf{A}(\lambda)^*)^{-1}\textbf{u}_f(\lambda) = \frac{\text{adj}(\textbf{A}(\lambda)^*)\textbf{u}_f(\lambda)}{\conj{\det \textbf{A}(\lambda)}}$$ is coanalytic. Moreover, the last expression shows that $\textbf{c}_f(\lambda)$ is in the coanalytic Smirnov class $\conj{N^+(\mathbb{C}^n)}$, and thus admits non-tangential limits almost everywhere on $\T$. For the boundary function we have the equality \begin{gather}
\textbf{c}_f(\z) = (\textbf{A}(\z)^*)^{-1}\textbf{u}_f(\z) = (\textbf{A}(\z)^*)^{-1}\textbf{B}(\z)^*f(\z) + \textbf{f}_1(\z) \label{eq51}
\end{gather} almost everywhere on $\T$. The following proposition summarizes the discussion above.

\begin{prop} \thlabel{cfprop}
	Let $\h[\textbf{B}]$ be of finite rank and $M_z$-invariant. If $f \in \h[\textbf{B}]$, then for all $\lambda \in \D$ we have $\frac{f(z)}{1-\conj{\lambda}z} \in \h[\textbf{B}]$ and $$J\frac{f(z)}{1-\conj{\lambda}z} = \Big(\frac{f(z)}{1-\conj{\lambda}z}, \frac{\textbf{f}_1(z) - \textbf{c}_f (\lambda)}{1-\conj{\lambda}z}\Big),$$ where $\textbf{c}_f$ is a coanalytic function of $\lambda$ which admits non-tangential boundary values almost everywhere on $\T$, and \eqref{eq51} holds.
\end{prop}

\subsection{Density of polynomials.} Since we are assuming that $1 \in \h[\textbf{B}]$, the $M_z$-invariance implies that the polynomials are contained in the space. 

\begin{thm}\thlabel{polydense}
	If finite rank $\h[\textbf{B}]$ is $M_z$-invariant, then the polynomials are dense in $\h$.
\end{thm}
\begin{proof} Since $J1 = (1, 0)$, an application of  \thref{cfprop} and \eqref{eq51} shows that
	\begin{equation} J\frac{1}{1-\conj{\lambda}z} = \Big(\frac{1}{1-\conj{\lambda}z}, -\frac{(\textbf{A}(\lambda)^*)^{-1}\textbf{B}(\lambda)^*1}{1-\conj{\lambda}z}\Big), \quad \lambda \in \D \label{eq55}
	\end{equation}  Assume that $f \in \h[\textbf{B}]$ is orthogonal to the polynomials. Then it follows that $Jf = (f, \textbf{f}_1)$ is orthogonal in $H^2 \oplus H^2(\mathbb{C}^n)$ to tuples of the form given by \eqref{eq55}. Thus
	\begin{gather*}
		0 = f(\lambda) - \ip{\textbf{f}_1(\lambda)}{(\textbf{A}(\lambda)^*)^{-1}\textbf{B}(\lambda)^*1}_{\mathbb{C}^n} \\ = f(\lambda)  - \textbf{B}(\lambda) (\textbf{A}(\lambda))^{-1}\textbf{f}_1(\lambda).
	\end{gather*} Set $$\textbf{h}(z) = (\textbf{A}(z))^{-1}\textbf{f}_1(z) = \frac{\text{adj}(\textbf{A}(z))\textbf{f}_1(z)}{\det \textbf{A}(z)}.$$ The last expression shows that $\textbf{h}$ belongs to the Smirnov class $N^+(\mathbb{C}^n)$. By the Smirnov maximum principle we furthermore have $\textbf{h} \in H^2(\mathbb{C}^n)$, since \begin{gather*}\int_{\T} \|\textbf{h}\|^2_{\mathbb{C}^n} d\m(\z) = \int_{\T} \|\textbf{B}\textbf{h}\|^2_{\mathbb{C}^n} d\m + \int_{\T} \|\textbf{A}\textbf{h}\|^2_{\mathbb{C}^n} d\m \\ = \int_{\T} |f|^2 d\m + \int_{\T} \|\textbf{f}_1\|^2_{\mathbb{C}^n} d\m = \|f\|^2_{\h[\textbf{B}]} < \infty.\end{gather*} Hence $Jf = (f,\textbf{f}_1) = (\textbf{B}\textbf{h}, \textbf{A} \textbf{h})$ with $\textbf{h} \in H^2(\mathbb{C}^n)$, so $Jf \in (J\h[\textbf{B}])^\perp$ and it follows that $f = 0$.
\end{proof} 
%

\subsection{Equivalent norms.} \label{isomorphismsubsec} In the finite rank case we meet a natural but fundamental question whether the rank of an $\h[\textbf{B}]$-space can be reduced with help of an equivalent norm. Here we assume that the space satisfies (A.1)-(A.3) with respect to the new norm. If that is the case, then the renormed space is itself an $\h[\textbf{D}]$-space for some function $\textbf{D}$. We will say that the spaces $\h[\textbf{B}]$ and $\h[\textbf{D}]$ are \textit{equivalent} if they are equal as sets and the two induced norms are equivalent. As mentioned in the introduction, \cite{ransforddbrdirichlet} shows that in the case $\h[\textbf{B}] = \mathcal{D}(\mu)$ for $\mu$ a finite set of point masses the space is equivalent to a rank one $\h(b)$-space. 

It is relatively easy to construct a $\textbf{B} = (b_1, b_2)$ such that $\h[\textbf{B}] = \h(b) \cap K_\theta$, where $b$ is non-extreme and $K_\theta = H^2 \ominus \theta H^2$. For appropriate choices of $b$ and $\theta$, the space $\h(b) \cap K_\theta$ cannot be equivalent to a de Branges-Rovnyak space. We shall not go into details of the construction, but we mention that it can be deduced from the proof of \thref{Linvsubspaces}. However, it is less obvious how to obtain an example of an $\h[\textbf{B}]$-space which is $M_z$-invariant and which is not equivalent to a non-extreme de Branges-Rovnyak space. The purpose of this section is to verify existence of such a space, and thus confirm that the class of $M_z$-invariant $\h[\textbf{B}]$-spaces is indeed a non-trivial extension of the $\h(b)$-spaces constructed from non-extreme $b$. The result that we shall prove is the following.

\begin{thm} \thlabel{nonsimilaritytheorem} For each integer $n \geq 1$ there exists $\textbf{B} = (b_1, \ldots, b_n)$ such that $\h[\textbf{B}]$ is $M_z$-invariant and has the following property: if $\h[\textbf{B}]$ is equivalent to a space $\h[\textbf{D}]$ with $\textbf{D} = (d_1, \ldots, d_m)$, then $m \geq n$. 
\end{thm}

Our first result in the direction of the proof of \thref{nonsimilaritytheorem} is interesting in its own right and gives a criterion for when two spaces $\h[\textbf{B}]$ and $\h[\textbf{D}]$ are equivalent. Below, $\conj{H^\infty}$ and $\conj{N^+}$ will denote the spaces of complex conjugates of functions in $H^\infty$ and $N^+$, respectively. 

\begin{thm} \thlabel{modulegen}
	Let $\h[\textbf{B}]$ with $\textbf{B} = (b_1, \ldots, b_n)$ and $\h[\textbf{D}]$ with $\textbf{D} = (d_1, \ldots, d_m)$ be two spaces with embeddings $J_1: \h[\textbf{B}] \to H^2 \oplus H^2(\mathbb{C}^n)$ and $J_2: \h[\textbf{D}] \to H^2 \oplus H^2(\mathbb{C}^m)$ as in \thref{modeltheoremanalytic}, such that \begin{gather*}
	J_1\frac{1}{1-\conj{\lambda}z} = \frac{(1,\textbf{c}(\lambda))}{1-\conj{\lambda}z} = \frac{(1, c_1(\lambda), \ldots, c_n(\lambda))}{1-\conj{\lambda}z}, \\ J_2\frac{1}{1-\conj{\lambda}z} = \frac{(1,\textbf{e}(\lambda))}{1-\conj{\lambda}z} = \frac{(1, e_1(\lambda), \ldots, e_m(\lambda))}{1-\conj{\lambda}z}.
	\end{gather*} Then the spaces $\h[\textbf{B}]$ and $\h[\textbf{D}]$ are equivalent if and only if the $\conj{H^\infty}$-submodules of $\conj{N^+}$ generated by $\{1, c_1, \ldots, c_n\}$ and by $\{1, e_1, \ldots, e_m\}$ coincide.
\end{thm}

\begin{proof}
Assume that $\h[\textbf{B}]$ and $\h[\textbf{D}]$ are equivalent. Let $i: \h[\textbf{B}] \to \h[\textbf{D}]$ be the identity mapping $if = f$. The subspaces $K_1 = J_1\h[\textbf{B}]$ and $K_2 = J_2\h[\textbf{D}]$ are invariant under the backward shift $L$ by \thref{modeltheoremanalytic}, and if $T = J_2iJ_1^{-1} : K_1 \to K_2$, then $LT = TL$. The commutant lifting theorem (see Theorem 10.29 and Exercise 10.31 of \cite{mccarthypick}) implies that $T$ extends to a Toeplitz operator  $T_{\boldsymbol{\Phi}}$, where $\boldsymbol{\Phi}$ is an $(m+1)$-by-$(n+1)$ matrix of bounded coanalytic functions, acting by matrix multiplication followed by the component-wise projection $P_+$ from $L^2$ onto $H^2$: \[ T_{\boldsymbol{\Phi}} (f, f_1, \ldots, f_n)^t = P_+\boldsymbol{\Phi}(f, f_1, \ldots, f_n)^t.\] Thus we have \begin{equation} \frac{(1,\textbf{e}(\lambda))^t}{1-\conj{\lambda}z} = T_{\boldsymbol{\Phi}}\frac{(1,\textbf{c}(\lambda))^t}{1-\conj{\lambda}z} = \frac{\boldsymbol{\Phi}(\lambda)(1, \textbf{c}(\lambda))^t}{1-\conj{\lambda}z} \label{eq600}. \end{equation} By reversing the roles of $\h[\textbf{B}]$ and $\h[\textbf{D}]$, we obtain again a coanalytic Toeplitz operator $T_{\boldsymbol{\Psi}}$ such that \begin{equation}\frac{(1,\textbf{c}(\lambda))^t}{1-\conj{\lambda}z} = T_{\boldsymbol{\Psi}}\frac{(1,\textbf{e}(\lambda))^t}{1-\conj{\lambda}z} = \frac{\boldsymbol{\Psi}(\lambda)(1, \textbf{e}(\lambda))^t}{1-\conj{\lambda}z}. \label{eq601} \end{equation} The equations \eqref{eq600} and \eqref{eq601} show that the $\conj{H^\infty}$-submodule of $\conj{N^+}$ generated by $\{1, c_1, \ldots, c_n\}$ coincides with the $\conj{H^\infty}$-submodule generated by $\{1, e_1, \ldots, e_m\}$. 
	
Conversely, assume that the $\conj{H^\infty}$-submodules of $\conj{N^+}$ generated by $\{1, c_1, \ldots, c_n\}$ and $\{1, e_1, \ldots, e_m\}$ coincide. Then there exists a matrix of coanalytic bounded functions $\boldsymbol{\Phi}$ such that \[\boldsymbol{\Phi}(1,c_1, \ldots, c_n)^t = (1, e_1, \ldots, e_m)^t\] where we can choose the top row of $\boldsymbol{\Phi}$ to equal $(1,0,\ldots, 0)$, and there exists also a matrix of coanalytic bounded functions $\boldsymbol{\Psi}$ such that \[\boldsymbol{\Psi}(1,e_1, \ldots, e_n)^t = (1, c_1, \ldots, c_n)^t\] with top row $(1, 0, \ldots, 0)$. It is now easy to see by density of the elements $\frac{1}{1-\conj{\lambda}z}$ in the spaces $\h[\textbf{B}]$ and $\h[\textbf{D}]$ (which can be deduced from \thref{polydense} and \thref{shiftspecrad}) that, in the notation of the above paragraph, we have $T_{\boldsymbol{\Phi}}K_1 = K_2$, $T_{\boldsymbol{\Psi}}K_2 = K_1$, $T_{\boldsymbol{\Psi}}T_{\boldsymbol{\Phi}} = I_{K_1}$, $T_{\boldsymbol{\Phi}}T_{\boldsymbol{\Psi}} = I_{K_2}$ and thus $\h[\textbf{B}]$ and $\h[\textbf{D}]$ are equivalent. 
	
\end{proof}

\begin{lemma} \thlabel{clemma} Let $\textbf{c} = (c_1, \ldots, c_n)$ be an arbitrary $n$-tuple of functions in  $\conj{N^+}$. There exists a $\textbf{B} = (b_1, \ldots, b_n)$ such that $\h[\textbf{B}]$ is $M_z$-invariant and an embedding $J: \h[\textbf{B}] \to H^2 \oplus H^2(\mathbb{C}^n)$ as in \thref{modeltheoremanalytic} such that \[ J\frac{1}{1-\conj{\lambda}z} = \Bigg( \frac{1}{1-\conj{\lambda}z}, \frac{\textbf{c}(\lambda)}{1-\conj{\lambda}z} \Bigg). \]
\end{lemma}

\begin{proof}
	Since $\conj{c_i} \in N^+$, there exists a factorization $\conj{c_i} = d_i/u_i$, where $d_i,u_i \in H^\infty$, and $u_i$ is outer. Let $\textbf{D} = (d_1, \ldots, d_n)$ and $\textbf{U} = \text{diag}(u_1, \ldots, u_n)$. The linear manifold \[ V = \{(\textbf{D}\mathbf{h}, \textbf{U}\mathbf{h}) : \mathbf{h} \in H^2(\mathbb{C}^n)\} \subset H^2 \oplus H^2(\mathbb{C}^n)\] is invariant under the forward shift $M_z$. It follows from general theory of shifts and the Beurling-Lax theorem (see Chapter 1 of \cite{rosenblumrovnyakhardyclasses}) that $M_z$ acting on $\text{clos}(V)$ is a shift of multiplicity $n$, and \[\text{clos}(V) = \{(\textbf{B}\mathbf{h}, \textbf{A}\mathbf{h}) : \mathbf{h} \in H^2(\mathbb{C}^n)\} \] for some analytic $\textbf{B}(z) = (b_1(z), \ldots, b_n(z))$ and $n$-by-$n$ matrix-valued analytic $\textbf{A}(z)$ such that the mapping $\textbf{h} \mapsto (\textbf{B}\textbf{h}, \textbf{A}\textbf{h})$ is an isometry from $H^2(\mathbb{C}^n)$ to $H^2 \oplus H^2(\mathbb{C}^n)$. It is easy to see that $\textbf{A}$ must be outer, since $\textbf{U}$ is, and that $\sum_{i=1}^n |b_i(z)|^2 \leq 1$ for all $z \in \D$. If $P$ is the projection from $H^2 \oplus H^2(\mathbb{C}^n)$ onto the first coordinate $H^2$, then as in the proof of \thref{modeltheorem} we set $\h[\textbf{B}]$ to be the image of $\text{clos}(V)^\perp$ under $P$, $\|f\|_{\h(\textbf{B}} = \|P^{-1}f\|_{H^2 \oplus H^2(\mathbb{C}^n)}$ and $J = P^{-1}$. The tuple $( \frac{1}{1-\conj{\lambda}z}, \frac{\textbf{c}(\lambda)}{1-\conj{\lambda}z} \Big)$ is obviously orthogonal to $V$, so $J\frac{1}{1-\conj{\lambda}z} = \Big( \frac{1}{1-\conj{\lambda}z}, \frac{\textbf{c}(\lambda)}{1-\conj{\lambda}z} \Big)$. The forward shift invariance of $\h[\textbf{B}]$ can be seen using the same argument as in the end of the proof of \thref{shiftinvcriterion}. 
\end{proof}

The next lemma is a version of a result of R. Mortini, see Lemma 2.8 and Theorem 2.9 of \cite{mortini}. The proofs in \cite{mortini} can be readily adapted to prove our version, we include a proof sketch for the convenience of the reader.

\begin{lemma} \thlabel{mingenideal}
	For each $n \geq 1$ there exists an outer function $u \in H^\infty$ and $f_1, \ldots, f_n \in H^\infty$ such that the ideal \[ \Big\{ g_0u + \sum_{i=1}^n g_if_i : g_0, g_1, \ldots, g_n \in H^\infty \Big\} \] cannot be generated by less than $n+1$ functions in $H^\infty$. 
\end{lemma}

\begin{proof} Let $\mathcal{M}$ be the maximal ideal space of $H^\infty$. The elements of $H^\infty$ are naturally functions on $\mathcal{M}$, and if $\xi \in \mathcal{M}$, then the evaluation of $f \in H^\infty$ at $\xi$ will be denoted by $f(\xi)$. Let $u$ be a bounded outer function and $I$ be an inner function such that $(u,I)$ is not a corona pair, so that there exists $\xi \in M$ such that $u(\xi) = I(\xi) = 0$. Let $f_k = I^k u^{n-k}$. We claim that the ideal generated by $\{ u^n, f_1, \ldots, f_n \}$ cannot be generated by less than $n+1$ functions. 

	The proof is split into two parts. In the first part we apply the idea contained in Lemma 2.8 of \cite{mortini} to verify the following claim: if $\phi_0, \phi_1, \ldots, \phi_n \in H^\infty$ are such that $\phi_0 u^n + \phi_1f_1 + \ldots \phi_nf_n = 0$, then $\phi_k(\xi) = 0$ for $0 \leq k \leq n$. To this end, the equality \[ \phi_0 u^n = - (\phi_1 u^{n-1}I + \ldots + \phi_n I^n)\] shows that $\phi_0$ is divisible by $I$, since the right-hand side is, but $u^n$, being outer, is not. It follows that $\phi_0 = Ih_0$ for some $h_0 \in H^\infty$, and therefore $\phi_0(\xi) = I(\xi)h_0(\xi) = 0$. Dividing the above equality by $I$ and re-arranging, we obtain \[(h_0u + \phi_1) u^{n-1} = -(\phi_2u^{n-2}I + \ldots + \phi_n I^{n-1}).\] As above, we must have $h_0u + \phi_1 = h_1I$, with $h_1 \in H^\infty$. Then \[\phi_1(\xi) = h_1(\xi)I(\xi) - h_0(\xi)u(\xi) = 0.\] By repeating the argument we conclude that $\phi_0(\xi) = \phi_1(\xi) = \ldots \phi_n(\xi) = 0$. 
	
	The second part of the proof is identical to the proof of Theorem 2.9 in \cite{mortini}. Assuming that the ideal generated by $\{u^n, f_1, \ldots, f_n\}$ is also generated by $\{e_1, \ldots, e_m\}$, we obtain a matrix $\textbf{M}$ of size $m $-by-$ (n+1)$ and a matrix $\textbf{N}$ of size $(n+1) $-by-$ m$, both with entires in $H^\infty$, such that \[ \textbf{M}(u^n, f_1, \ldots, f_n)^t = (e_1, \ldots, e_m)^t, \quad \textbf{N}(e_1, \ldots, e_m)^t = (u^n, f_1, \ldots,  f_n)^t.\] Then $(\textbf{NM}- I_{n+1})(u^n, f_1, \ldots, f_n)^t = \textbf{0}$, where $I_{n+1}$ is the identity matrix of dimension $n+1$. By the first part of the proof we obtain that $\textbf{N}(\xi)\textbf{M}(\xi) = I_{n+1}$, where the evaluation of the matrices at $\xi$ is done entrywise. Then the rank of the matrix $\textbf{N}(\xi)$ is at least $n+1$, i.e., $m \geq n+1$. 	
\end{proof}

We are ready to prove the main result of the section.
\begin{proof}[Proof of \thref{nonsimilaritytheorem}]
	By \thref{mingenideal} there exist $u, f_1, \ldots, f_n \in H^\infty$, with $u$ outer, such that the ideal of $H^\infty$ generated by $\{u, f_1, \ldots, f_n\}$ is not generated by any set of size less than $n+1$. Let $c_i = \conj{f_i/u} \in \conj{N^+}$ and apply \thref{clemma} to $\textbf{c} = (c_1, \ldots, c_n)$ to obtain a $\textbf{B} = (b_1, \ldots, b_n)$ and a space $\h[\textbf{B}]$ such that \[ J\frac{1}{1-\conj{\lambda}z} = \Bigg( \frac{1}{1-\conj{\lambda}z}, \frac{\textbf{c}(\lambda)}{1-\conj{\lambda}z} \Bigg). \] If $\h[\textbf{B}]$ is equivalent to $\h[\textbf{D}]$, where $\textbf{D} = (d_1, \ldots, d_m)$ and  \[ J_2\frac{1}{1-\conj{\lambda}z} = \frac{(1,\textbf{e}(\lambda))}{1-\conj{\lambda}z} = \frac{(1, e_1(\lambda), \ldots, e_m(\lambda))}{1-\conj{\lambda}z},\] where $J_2$ is the embedding associated to $\h[\textbf{D}]$, then \thref{modulegen} implies that the sets $\{u, f_1, \ldots, f_n \}$ and $\{u, \conj{e_1} u, \ldots, \conj{e_m} u \}$ generate the same ideal in $H^\infty$. Thus $n+1 \geq m+1$.
\end{proof}

\subsection{$M_z$-invariant subspaces.}

Our structure theorem for $M_z$-invariant subspaces will follow easily from \thref{modeltheoremanalytic} after this preliminary lemma.

\begin{lemma} \thlabel{index1prop} Let $\h[\textbf{B}]$ be of finite rank and $M_z$-invariant. If $\M$ is an $M_z$-invariant subspace of $\h[\textbf{B}]$, then for each $\lambda \in \D$ we have that $\dim \M \ominus (M_z -\lambda)\M = 1$, and thus $M$ is nearly invariant.
\end{lemma}

\begin{proof}
Let $f \in \M, h \in \M^\perp$ and  $Jf = (f, \textbf{f}_1), Jh = (h, \textbf{h}_1)$. By $M_z$-invariance of $\M$ we have, in the notation of \thref{cfprop}, \begin{equation}
0 = \ip{(I_{\h[\textbf{B}]}-\conj{\lambda} M_z)^{-1}f}{h}_{\h[\textbf{B}]} = \int_\T \frac{f(\z)\conj{h(\z)} + \ip{\textbf{f}_1(\z)}{\textbf{h}_1(\z)}_{\mathbb{C}^n}}{1-\conj{\lambda}z}d\m(\z) - \ip{\textbf{c}_f(\lambda)}{\textbf{h}_1(\lambda)}_{\mathbb{C}^n}. \label{eq541}
\end{equation} Let $K_{f,h}$ be the Cauchy transform \begin{equation}K_{f,h}(\lambda) = \lambda \int_\T \frac{f(\z)\conj{h(\z)} + \ip{\textbf{f}_1(\z)}{\textbf{h}_1(\z)}_{\mathbb{C}^n}}{z-\lambda} d\m(\z). \label{eq542}\end{equation} Then $K_{f,h}$ is analytic for $\lambda \in \D$ and admits non-tangential boundary values on $\T$. Adding \eqref{eq541} and \eqref{eq542} gives \begin{equation*}K_{f,h}(\lambda) = \int_\T \frac{1-|\lambda|^2}{|\z - \lambda|^2} \big( f(\z)\conj{h(\z)} + \ip{\textbf{f}_1(\z)}{\textbf{h}_1(\z)}_{\mathbb{C}^n}\big) d\m(\z) - \ip{\textbf{c}_f(\lambda)}{\textbf{h}_1(\lambda)}_{\mathbb{C}^n}. \end{equation*} By taking the limit $|\lambda| \to 1$ and using basic properties of Poisson integrals we see that, for almost every $\lambda \in \T$, we have the equality \begin{gather*}
K_{f,h}(\lambda) = f(\lambda)\conj{h(\lambda)} + \ip{\textbf{f}_1(\lambda)}{\textbf{h}_1(\lambda)}_{\mathbb{C}^n} - \ip{\textbf{c}_f(\lambda)}{\textbf{h}_1(\lambda)}_{\mathbb{C}^n} \\ = f(\lambda)\conj{h(\lambda)} - \ip{(\textbf{A}(\lambda)^*)^{-1}\textbf{B}(\lambda)^*f(\lambda)}{\textbf{h}_1(\lambda)}_{\mathbb{C}^n} \\ =
f(\lambda)\Big( \conj{h(\lambda)} - \ip{(\textbf{A}(\lambda)^*)^{-1}\textbf{B}(\lambda)^*1}{\textbf{h}_1(\lambda)}_{\mathbb{C}^n}\Big),
\end{gather*} where we used \eqref{eq51} in the computation. The meromorphic function $K_{f,h}/f$ thus depends only on $h$, and not on $f$. 

Let $f(\lambda) = 0$ for some $\lambda \in \D \setminus \{0\}$ which is not a common zero of $\M$, so that there exists $g \in \M$ with $g(\lambda) \neq 0$. From $K_{f,h}/f = K_{g,h}/g$ we deduce that $K_{f,h}(\lambda)g(\lambda) = K_{g,h}(\lambda)f(\lambda) = 0$, and thus \begin{gather*}
0 = K_{f,h}(\lambda) = \lambda \int_\T \frac{f(\z)\conj{h(\z)} + \ip{\textbf{f}_1(\z)}{\textbf{h}_1(\z)}_{\mathbb{C}^n}}{z-\lambda} d\m(\z) \\ = \lambda \int_\T \frac{f(\z)\conj{h(\z)} + \ip{\textbf{f}_1(\z) - \textbf{f}_1(\lambda)}{\textbf{h}_1(\z)}_{\mathbb{C}^n}}{z-\lambda} d\m(\z) \\ = \lambda \ip{\tfrac{f(z)}{z-\lambda}}{h}_{\h[\textbf{B}]}.
\end{gather*} Since $h \in \M^\perp$ is arbitrary, we conclude that $\tfrac{f(z)}{z-\lambda} \in \M$, and thus $\dim \M \ominus (M_z - \lambda)\M = 1$. The fact that $\dim \M \ominus M_z\M = 1$ holds also for $\lambda = 0$ or a common zero of the functions in $\M$ follows from basic Fredholm theory. The operators $M_z - \lambda$ are injective semi-Fredholm operators, and thus $\lambda \mapsto \dim \M \ominus (M_z - \lambda)\M$ is a constant function in $\D$. 
\end{proof}

The following is our main theorem on $M_z$-invariant subspaces of finite rank $\h[\textbf{B}]$-spaces. 

\begin{thm} Let $\h = \h[\textbf{B}]$ be of finite rank and  $M_z$-invariant and $\M$ be a closed $M_z$-invariant subspace of $\h$. Then

\begin{enumerate}[(i)]
\item $\dim \M \ominus M_z\M = 1$,
\item any non-zero element in $\M \ominus M_z\M$ is a cyclic vector for $M_z|\M$,
\item if $\phi \in \M \ominus M_z \M$ is of norm $1$, then there exists a space $\h[\textbf{C}]$ invariant under $M_z$, where $\textbf{C} = (c_1, \ldots, c_k)$ and $k \leq n$, such that \begin{equation*}\M = \phi \h[\textbf{C}] \end{equation*} and the mapping $g \mapsto \phi g$ is an isometry from $\h[\textbf{C}]$ onto $\M$, 
\item if $J$ is the embedding given by \thref{modeltheoremanalytic}, $\phi \in \M \ominus M_z\M$ with $J\phi = (\phi, \boldsymbol{\phi}_1)$, then \[ \M = \big\{ f \in \h[\textbf{B}] : \tfrac{f}{\phi} \in H^2, \tfrac{f}{\phi}\boldsymbol{\phi_1} \in H^2(\mathbb{C}^n) \big\}.\]

\end{enumerate}
\end{thm}

\begin{proof}
Part (i) has been established \thref{index1prop} and part (ii) follows from (iii) by \thref{polydense}. It thus suffices to prove parts (iii) and (iv). 

We verified in \thref{index1prop} that $\M$ is nearly invariant, and thus norm formula \eqref{eq221} of \thref{nearinvnorm} applies. A computation shows that, in the notation of \thref{nearinvnorm}, we have $L^\phi_\lambda f = L_\lambda (f-\tfrac{f(\lambda)}{\phi(\lambda)}\phi)$, at least when $\phi(\lambda) \neq 0$. Thus if $Jf = (f, \textbf{f}_1)$ and $J\phi =(\phi, \boldsymbol{\phi}_1)$, then by \eqref{eq221} and part (i) of \thref{modelformulaconnection} we obtain that \begin{equation}\|f\|^2_{\h[\textbf{B}]} = \|f/\phi\|^2_{H^2} + \|\textbf{g}_1\|^2_{H^2(\mathbb{C}^n)}\label{501} \end{equation} where $\textbf{g}_1(z) = \textbf{f}_1(z) - \tfrac{f(z)}{\phi(z)}\boldsymbol{\phi}_1(z)$. The mapping $If := (f/\phi, \textbf{g}_1)$ is therefore an isometry from $\M$ into $H^2 \oplus H^2(\mathbb{C}^n)$. The identity $IL^\phi f = (L(f/\phi), L\textbf{g}_1)$ shows that $I\M$ is a backward shift invariant subspace of $H^2 \oplus H^2(\mathbb{C}^n) \simeq H^2(\mathbb{C}^{n+1})$. Consequently by the Beurling-Lax theorem we have that $(I\M)^\perp = \Psi H^2(\mathbb{C}^k)$, for some $(n+1)$-by-$k$ matrix-valued bounded analytic function such that $\Psi(\z): \mathbb{C}^k \to \mathbb{C}^{n+1}$ is an isometry for almost every $\z \in \T$. We claim that $k \leq n$. Indeed, in other case $\Psi$ is an $(n+1)$-by-$(n+1)$ square matrix, and hence $\psi(z) = \det \Psi (z)$ is a non-zero inner function. We would then obtain \begin{equation}\psi H^2(\mathbb{C}^{n+1}) =  \Psi \text{adj}(\Psi) H^2(\mathbb{C}^{n+1}) \subset \Psi H^2(\mathbb{C}^{n+1}) = (I\M)^\perp. \label{eq59} \end{equation} But since $\M$ is shift invariant, the function $p\phi$ is contained in $\M$ for any polynomial $p$, and hence for any polynomial $p$ there exists a tuple of the form $(p, \textbf{g})$ in $I\M$. Together with $
\eqref{eq59}$ this shows that the polynomials are orthogonal to $\psi H^2$, so $\psi = 0$ and we arrive at a contradiction. 

Decompose the matrix $\Psi$ as \begin{equation*}
\Psi(z) = \begin{bmatrix}
\textbf{C}(z) \\\textbf{D}(z)
\end{bmatrix}
\end{equation*} where $\textbf{C}(z) = (c_1(z), \ldots, c_k(z))$ and $\textbf{D}(z)$ is an $n$-by-$k$ matrix. Consider the Hilbert space $\tilde{\M} = \M/\phi = \{f/\phi : f \in \M\}$ with the norm $\|f/\phi\|_{\tilde{\M}} = \|f\|_{\h[\textbf{B}]}$. By \eqref{501}, the map \begin{equation} \tilde{I}f/\phi := \Big(f/\phi, \textbf{f}_1 - \tfrac{f}{\phi}\boldsymbol{\phi}_1 \Big) \label{Iemb} \end{equation} is an isometry from $\tilde{\M}$ into $H^2 \oplus H^2(\mathbb{C}^n)$, and \begin{equation}(\tilde{I}\tilde{\M})^\perp = \{ (\textbf{C}\textbf{h}, \textbf{D}\textbf{h}) : \textbf{h} \in H^2(\mathbb{C}^k) \}. \label{eqCD} \end{equation} The argument of \thref{modeltheorem} can be used to see that $$k_{\tilde{\M}}(\lambda,z) = \frac{1- \sum_{i=1}^k \conj{c_i(\lambda)}c_i(z)}{1-\conj{\lambda}z}$$ is the reproducing kernel of $\tilde{\M}$. Thus $\tilde{\M} = \h[\textbf{C}]$, and the proof of part (iii) is complete.

Finally, we prove part (iv). The inclusion of $\M$ in the set given in (iv) has been established in the proof of part (iii) above. On the other hand, assume that $f \in \h[\textbf{B}]$ is contained in that set. We will show that $(f/\phi, \textbf{f}_1 - \tfrac{f}{\phi}\boldsymbol{\phi}_1 )$ is orthogonal to the set given in \eqref{eqCD}, and thus $f/\phi \in \tilde{\M}$, so that $f \in \M$. In order to verify the orthogonality claim, we must show that $\textbf{C}^*\tfrac{f}{\phi} + \textbf{D}^*(\textbf{f}_1 - \tfrac{f}{\phi}\boldsymbol{\phi}_1) \in \overline{H^2_0(\mathbb{C}^n)}.$ According to \thref{cfprop}, we have that  \[ J\frac{\phi(z)}{1-\conj{\lambda}z} = \Big( \frac{\phi(z)}{1-\conj{\lambda}z}, \frac{\boldsymbol{\phi}_1(z) - \textbf{c}_\phi(\lambda)}{1-\conj{\lambda}z} \Big) \] for some coanalytic function $\textbf{c}_\phi$ which satisfies $\textbf{c}_\phi = (\textbf{A}^*)^{-1}\textbf{B}^*\phi + \boldsymbol{\phi}_1$ on $\T$. Setting $f(z) = \frac{\phi(z)}{1-\conj{\lambda}z}$ in \eqref{Iemb} we obtain from \eqref{eqCD} that \[ \Big( \frac{1}{1-\conj{\lambda}z} ,  -\frac{\textbf{c}_\phi(\lambda)}{1-\conj{\lambda}z}\Big) \perp \{ (\textbf{C}\textbf{h}, \textbf{D}\textbf{h}) : \textbf{h} \in H^2(\mathbb{C}^k) \}\] and then it easily follows that $\textbf{c}_\phi(\lambda) = (\textbf{D}^*(\lambda))^{-1}\textbf{C}^*(\lambda)1$. Using \eqref{eq51} we obtain the boundary value equality \[\textbf{C}^* = \textbf{D}^*(\textbf{A}^*)^{-1}\textbf{B}^* \phi + \textbf{D}^*\boldsymbol{\phi}_1,\] and thus on $\T$ we have \begin{equation}\textbf{C}^*\tfrac{f}{\phi} + \textbf{D}^*(\textbf{f}_1 - \tfrac{f}{\phi}\boldsymbol{\phi}_1) = \textbf{D}^* (\textbf{A}^*)^{-1}(\textbf{B}^*f + \textbf{A}^* \textbf{f}_1) \label{eq99}. \end{equation} Since $f \in \h[\textbf{B}]$ we have that $\textbf{B}^*f +  \textbf{A}^*g \in \overline{H^2_0(\mathbb{C}^n)}$ and thus \eqref{eq99} represents square-integrable boundary function of a coanalytic function in the Smirnov class. An appeal to the Smirnov maximum principle completes the proof of (iv).
\end{proof}

\subsection{Backward shift invariant subspaces.} The lattice of $L$-invariant subspaces of $\h[\textbf{B}]$-spaces is much less complicated than the lattice of $M_z$-invariant subspaces. The following theorem generalizes a result of \cite{sarasondoubly} for $\h(b)$ with non-extreme $b$. Our method of proof is new, and relies crucially on \thref{modeltheoremanalytic}. 

\begin{thm} \thlabel{Linvsubspaces}Any proper $L$-invariant subspace of a $M_z$-invariant finite rank $\h[\textbf{B}]$-space is of the form $$\h[\textbf{B}] \cap K_\theta,$$ where $\theta$ is an inner function and $K_\theta = H^2 \ominus \theta H^2$. 
\end{thm}

\begin{proof}
Let $J: \h[\textbf{B}] \to H^2 \oplus H^2({\mathbb{C}^n})$ be the embedding of \thref{modeltheoremanalytic}. If $\M$ is an $L$-invariant subspace of $\h[\textbf{B}]$, then $J\M$ is an $L$-invariant subspace of $H^2 \oplus H^2(\mathbb{C}^n)$ by (iii) of \thref{modeltheoremanalytic}. Thus $(J\M)^\perp$ is an $M_z$-invariant subspace containing $(J\h[\textbf{B}])^\perp = \{(\textbf{B}\textbf{h}, \textbf{A}\textbf{h}) : \textbf{h} \in H^2(\mathbb{C}^n)\}$. Because $M_z$ acting on $(J\h[\textbf{B}])^\perp$ is a shift of multiplicity $n$, the multiplicity of $M_z$ acting on $(J\M)^\perp$ is at least $n$, and since $(J\M)^\perp \subset H^2 \oplus H^2(\mathbb{C}^n)$, it is at most $n+1$. We claim that this multiplicity must equal $n+1$. Indeed, if it was equal to $n$, then it is easy to see that $$(J\M)^\perp = \{ (\textbf{Ch}, \textbf{Dh}) : \textbf{h} \in H^2(\mathbb{C}^n)\}$$ for some $\textbf{C}(z) = (c_1(z), \ldots, c_n(z))$ and $\textbf{D}(z)$ an $n$-by-$n$-matrix valued analytic function. The fact that no tuple of the form $(0,\textbf{g})$ is included in $J\M$ implies that $\textbf{D}$ is an outer function, and thus $\textbf{D}(\lambda)$ is an invertible operator for every $\lambda \in \D$. The tuple \begin{equation*} \Big(\frac{1}{1-\conj{\lambda}z}, -\frac{(\textbf{D}(\lambda)^*)^{-1}\textbf{C}(\lambda)^*1}{1-\conj{\lambda}z}\Big)
\end{equation*}  is clearly orthogonal to $(J\M)^\perp$, and thus $\frac{1}{1-\conj{\lambda}z} \in \M$ for every $\lambda \in \D$. Then $\M = \h[\textbf{B}]$ by the proof of \thref{polydense}. We assumed that $\M$ is a proper subspace, and so multiplicity of $M_z$ on $(J\M)^\perp$ cannot be $n$.

Having established that $M_z$ is a shift of multiplicity $n+1$ on $(J\M)^\perp$, we conclude that $$(J\M)^\perp= \Psi H^2(\mathbb{C}^{n+1}),$$ where $\Psi$ is an $(n+1)$-by-$(n+1)$ matrix-valued inner function, and $\theta = \det \Psi$ is a non-zero scalar-valued inner function. Note that $\theta H^2(\mathbb{C}^{n+1}) = \Psi\,\text{adj}(\Psi) H^2(\mathbb{C}^{n+1}) \subseteq \Psi H^2(\mathbb{C}^{n+1})$. Thus if $f \in \M$, then $Jf = (f,\textbf{f}_1) \perp \Psi H^2(\mathbb{C}^{n+1}) \supseteq \theta H^2(\mathbb{C}^{n+1})$. It follows that $f \in K_\theta$, and thus we have shown that $\M \subseteq \h[\textbf{B}] \cap K_\theta$. Next, consider the $(n+1)$-by-$(n+1)$ matrix \begin{gather*}
\textbf{M}(z) = \begin{bmatrix}
\textbf{B}(z) & \theta(z) \\ \textbf{A}(z) & 0
\end{bmatrix}.
\end{gather*} Then it is easy to see that $f \in \h[\textbf{B}] \cap K_\theta$ if and only if $Jf = (f,\textbf{f}_1) \perp \textbf{M}(z)\textbf{h}(z)$ for all $\textbf{h} \in H^2(\mathbb{C}^{n+1})$. If $\textbf{M}(z) = \textbf{I}(z)\textbf{U}(z)$ is the inner-outer factorization of $\textbf{M}$ into an $(n+1)$-by-$(n+1)$-matrix valued inner function $\textbf{I}$ and an $(n+1)$-by-$(n+1)$-matrix valued outer function $\textbf{U}$, then we also have that \begin{equation}f \in \h[\textbf{B}] \cap K_\theta \text{ if and only if } Jf = (f, \textbf{f}_1) \perp \textbf{I}H^2(\mathbb{C}^{n+1}). \label{eq72} \end{equation} From the containment $\M \subseteq \h[\textbf{B}] \cap K_\theta$ we get by taking orthocomplements that $\textbf{I}H^2(\mathbb{C}^{n+1}) \subseteq \Psi H^2(\mathbb{C}^{n+1})$ and thus there exists a factorization $\textbf{I} = \Psi \textbf{J}$, where $\textbf{J}$ is an $(n + 1)$-by-$(n+1)$-matrix valued inner function. Since $-\theta \det \textbf{A} = \det \textbf{M} = \det \textbf{I} \det \textbf{U}$, we see (by comparing inner and outer factors) that $\det \textbf{I} = \alpha\theta$, with $\alpha \in \T$, and so $\alpha \theta = \det \textbf{I} = \det \Psi \det \textbf{J} = \theta \det \textbf{J}$. We conclude that $\det \textbf{J}$ is a constant, and thus $\textbf{J}$ is a constant unitary matrix. But then $(J\M)^\perp = \Psi H^2(\mathbb{C}^{n+1}) = \textbf{I}H^2(\mathbb{C}^{n+1})$, and so the claim follows by \eqref{eq72}.
\end{proof}

\bibliographystyle{siam}
\bibliography{mybib}

\begin{thebibliography}{10}

\bibitem{mccarthypick}
{\sc J.~Agler and J.~E. McCarthy}, {\em Pick Interpolation and Hilbert Function
  Spaces}, vol.~44 of Graduate Studies in Mathematics, American Mathematical
  Society, 2002.

\bibitem{aleksandrovinv}
{\sc A.~B. Aleksandrov}, {\em Invariant subspaces of shift operators. {A}n
  axiomatic approach}, Zap. Nauchn. Sem. Leningrad. Otdel. Mat. Inst. Steklov.
  (LOMI), 113 (1981), pp.~7--26, 264.

\bibitem{alemanhabil}
{\sc A.~Aleman}, {\em The multiplication operator on {H}ilbert spaces of
  analytic functions}, Habilitationsschrift, Fern Universität, Hagen,  (1993).

\bibitem{acppredualsqp}
{\sc A.~Aleman, M.~Carlsson, and A.-M. Persson}, {\em Preduals of
  {$Q_p$}-spaces}, Complex Var. Elliptic Equ., 52 (2007), pp.~605--628.

\bibitem{afr}
{\sc A.~Aleman, N.~Feldman, and W.~Ross}, {\em The {H}ardy space of a slit
  domain}, Frontiers in Mathematics, Birkh\"auser Verlag, Basel, 2009.

\bibitem{dbrcont}
{\sc A.~Aleman and B.~Malman}, {\em Density of disk algebra functions in de
  {B}ranges--{R}ovnyak spaces}, C. R. Math. Acad. Sci. Paris, 355 (2017),
  pp.~871--875.

\bibitem{alemanrichtersimplyinvariant}
{\sc A.~Aleman and S.~Richter}, {\em Simply invariant subspaces of {$H^2$} of
  some multiply connected regions}, Integral Equations Operator Theory, 24
  (1996), pp.~127--155.

\bibitem{ars}
{\sc A.~Aleman, S.~Richter, and C.~Sundberg}, {\em Beurling's theorem for the
  {B}ergman space}, Acta Math., 177 (1996), pp.~275--310.

\bibitem{ballbolotnikov}
{\sc J.~Ball and V.~Bolotnikov}, {\em de {B}ranges-{R}ovnyak spaces: basics and
  theory}, pp.~631--679.

\bibitem{revcarlesonross}
{\sc A.~Blandign\`eres, E.~Fricain, F.~Gaunard, A.~Hartmann, and W.~T. Ross},
  {\em Direct and reverse {C}arleson measures for {$H(b)$} spaces}, Indiana
  Univ. Math. J., 64 (2015), pp.~1027--1057.

\bibitem{borichevpriv}
{\sc A.~Borichev}.
\newblock private communication.

\bibitem{borichevhedenmalmcyclicity}
{\sc A.~Borichev and H.~Hedenmalm}, {\em Harmonic functions of maximal growth:
  invertibility and cyclicity in {B}ergman spaces}, J. Amer. Math. Soc., 10
  (1997), pp.~761--796.

\bibitem{cauchytransform}
{\sc J.~Cima, A.~Matheson, and W.~T. Ross}, {\em The {C}auchy transform},
  vol.~125 of Mathematical Surveys and Monographs, American Mathematical
  Society, Providence, RI, 2006.

\bibitem{ransforddbrdirichlet}
{\sc C.~Costara and T.~Ransford}, {\em Which de {B}ranges-{R}ovnyak spaces are
  {D}irichlet spaces (and vice versa)?}, J. Funct. Anal., 265 (2013),
  pp.~3204--3218.

\bibitem{durenbergmanspaces}
{\sc P.~Duren and A.~Schuster}, {\em Bergman spaces}, vol.~100 of Mathematical
  Surveys and Monographs, American Mathematical Society, Providence, RI, 2004.

\bibitem{esterle}
{\sc J.~Esterle}, {\em Toeplitz operators on weighted {H}ardy spaces}, Algebra
  i Analiz, 14 (2002), pp.~92--116.

\bibitem{revcarlmodelspaces}
{\sc A.~Hartmann, X.~Massaneda, A.~Nicolau, and J.~Ortega-Cerd\`a}, {\em
  Reverse {C}arleson measures in {H}ardy spaces}, Collect. Math., 65 (2014),
  pp.~357--365.

\bibitem{havinbook}
{\sc V.~P. Havin and B.~J\"oricke}, {\em The uncertainty principle in harmonic
  analysis}, vol.~72 of Encyclopaedia Math. Sci., Springer, Berlin, 1995.

\bibitem{hedenmalmbergmanspaces}
{\sc H.~Hedenmalm, B.~Korenblum, and K.~Zhu}, {\em Theory of {B}ergman spaces},
  vol.~199 of Graduate Texts in Mathematics, Springer-Verlag, New York, 2000.

\bibitem{helsonbook}
{\sc H.~Helson}, {\em Lectures on invariant subspaces}, Academic Press, New
  York-London, 1964.

\bibitem{hittanulus}
{\sc D.~Hitt}, {\em Invariant subspaces of {$H^2$} of an annulus}, Pacific J.
  Math., 134 (1988), pp.~101--120.

\bibitem{mortini}
{\sc R.~Mortini}, {\em Generating sets for ideals of finite type in
  {$H^\infty$}}, Bull. Sci. Math., 136 (2012), pp.~687--708.

\bibitem{nikolskiivasyuninnotesfuncmod}
{\sc N.~K. Nikolski and V.~I. Vasyunin}, {\em Notes on two function models}, in
  The {B}ieberbach conjecture ({W}est {L}afayette, {I}nd., 1985), vol.~21 of
  Math. Surveys Monogr., Amer. Math. Soc., Providence, RI, 1986, pp.~113--141.

\bibitem{richterreinedirichlet}
{\sc S.~Richter}, {\em Invariant subspaces of the {D}irichlet shift}, J. Reine
  Angew. Math., 386 (1988), pp.~205--220.

\bibitem{richtersundberglocaldirichlet}
{\sc S.~Richter and C.~Sundberg}, {\em A formula for the local {D}irichlet
  integral}, Michigan Math. J., 38 (1991), pp.~355--379.

\bibitem{rosenblumrovnyakhardyclasses}
{\sc M.~Rosenblum and J.~Rovnyak}, {\em Hardy classes and operator theory},
  Oxford Mathematical Monographs, The Clarendon Press, Oxford University Press,
  New York, 1985.

\bibitem{sarasondoubly}
{\sc D.~Sarason}, {\em Doubly shift-invariant spaces in {$H^2$}}, J. Operator
  Theory, 16 (1986), pp.~75--97.

\bibitem{sarasonbook}
{\sc D.~Sarason}, {\em Sub-{H}ardy {H}ilbert spaces in the unit disk}, vol.~10
  of University of Arkansas Lecture Notes in the Mathematical Sciences, John
  Wiley \& Sons, Inc., New York, 1994.

\bibitem{nagyfoiasharmop}
{\sc B.~Sz.-Nagy, C.~Foias, H.~Bercovici, and L.~K\'erchy}, {\em Harmonic
  analysis of operators on {H}ilbert space}, Universitext, Springer, New York,
  second~ed., 2010.

\bibitem{vinogradovthm}
{\sc S.~A. Vinogradov}, {\em Properties of multipliers of integrals of
  {C}auchy-{S}tieltjes type, and some problems of factorization of analytic
  functions}, in Mathematical programming and related questions ({P}roc.
  {S}eventh {W}inter {S}chool, {D}rogobych, 1974), {T}heory of functions and
  functional analysis ({R}ussian), Central \`Ekonom.-Mat. Inst. Akad. Nauk
  SSSR, Moscow, 1976, pp.~5--39.
\newblock English translation in Transl. Amer. Math. Soc. (2) 115 (1980) 1–32).

\end{thebibliography}

\end{document}